\documentclass[11pt]{amsart}
\usepackage{setspace}
\usepackage{amsmath, amsfonts, amsthm, amstext, amssymb, xspace, float, mathtools}
\usepackage{amssymb}
\usepackage{newtxtext}
\usepackage[all,2cell]{xy}
\usepackage{verbatim}
\usepackage{hyperref}
\usepackage{color}
\usepackage[normalem]{ulem}
\usepackage[usenames,dvipsnames]{xcolor}
\usepackage{graphicx}							
\usepackage{tikz}
\usetikzlibrary{arrows.meta}
\usetikzlibrary{matrix,arrows,cd,decorations.pathreplacing}
\usepackage{quiver}
\usepackage{xspace}
\usepackage[shortlabels]{enumitem}
\usepackage[capitalise, nameinlink]{cleveref}

\usepackage[
textwidth=2.8cm,
textsize=small,
colorinlistoftodos]
{todonotes}
\newtheorem{ThmAlpha}{Theorem}

\newtheorem{CorAlpha}{Corollary}

\DeclareMathSymbol{\ssquare}{\mathbin}{AMSa}{"03}
\DeclareMathSymbol{\square}{\mathbin}{AMSa}{"03}
\DeclareMathSymbol{\Box}{\mathbin}{AMSa}{"03}

\newcommand{\category}{{$\infty$-cat\-e\-go\-ry}\xspace}
\newcommand{\categories}{$\infty$-cat\-e\-gories\xspace}

\def\t{\mathbb{T}}

\def\c{\mathbb{C}}

\newcommand{\mft}{\underline{\mathbb{F}}_2}

\newcommand{\res}{\mathrm{res}}

\newcommand{\colim}{\operatorname*{colim}}

\newcommand{\TT}{\mathbb{T}}
\newcommand{\zz}{\mathbb{Z}}
\newcommand{\CP}{\mathbb{C}P}

\DeclareMathOperator{\THH}{THH}
\DeclareMathOperator{\TC}{TC}
\DeclareMathOperator{\Sp}{Sp}
\DeclareMathOperator{\Top}{Top}
\DeclareMathOperator{\TP}{TP}
\DeclareMathOperator{\Fun}{Fun}
\DeclareMathOperator{\MUR}{\operatorname{MU_\mathbb{R}}}

\DeclareMathOperator{\BPR}{\operatorname{BP_\mathbb{R}}}
\DeclareMathOperator{\KR}{\operatorname{KR}}
\DeclareMathOperator{\TCR}{TCR}

\newcommand{\cal}[1]{\mathcal{#1}}

\newcommand{\CAlg}{\mathrm{CAlg}}

\newcommand{\Fin}{\mathrm{Fin}}

\newcommand{\coCAlg}{\mathrm{coCAlg}}

\newcommand{\uSp}{\underline{\mathrm{Sp}}}

\newcommand{\op}{\mathrm{op}}
\newcommand{\Hom}{\mathrm{Hom}}

\newcommand{\F}{\mathbb{F}}

\newcommand{\cofib}{\mathrm{cofib}}

\newcommand{\id}{\mathrm{id}}

\newcommand{\THR}{\operatorname{THR}}
\newcommand{\TPR}{\operatorname{TPR}}

\theoremstyle{plain}   
\newtheorem{thm}{Theorem}[section] 
\newtheorem{corollary}[thm]{Corollary}
\newtheorem{lemma}[thm]{Lemma}
\newtheorem{proposition}[thm]{Proposition}

\newtheorem{example}[thm]{Example}

\theoremstyle{definition}
\newtheorem{definition}[thm]{Definition}
\newtheorem{construction}[thm]{Construction}

\newtheorem{warning}[thm]{Warning}

\theoremstyle{remark}
\newtheorem{remark}[thm]{Remark}

\newtheorem{notation}[thm]{Notation}

\makeatletter
\let\c@equation\c@thm

\makeatother

\makeatletter
\newcommand*{\relrelbarsep}{.386ex}
\newcommand*{\relrelbar}{%
  \mathrel{%
    \mathpalette\@relrelbar\relrelbarsep
  }%
}
\newcommand*{\@relrelbar}[2]{%
  \raise#2\hbox to 0pt{$\m@th#1\relbar$\hss}%
  \lower#2\hbox{$\m@th#1\relbar$}%
}
\providecommand*{\rightrightarrowsfill@}{%
  \arrowfill@\relrelbar\relrelbar\rightrightarrows
}
\providecommand*{\leftleftarrowsfill@}{%
  \arrowfill@\leftleftarrows\relrelbar\relrelbar
}
\providecommand*{\xrightrightarrows}[2][]{%
  \ext@arrow 0359\rightrightarrowsfill@{#1}{#2}%
}
\providecommand*{\xleftleftarrows}[2][]{%
  \ext@arrow 3095\leftleftarrowsfill@{#1}{#2}%
}
\makeatother

\usepackage{xcolor}
\definecolor{seagreen}{RGB}{46,139,87}
\definecolor{maroon}{RGB}{128,0,0}
\definecolor{darkviolet}{RGB}{210,150,180}

\UseAllTwocells

\title{On homological Real trace methods}
\author[M. Cho]{Myungsin Cho}
\author[T. Gerhardt]{Teena Gerhardt}
\author[L. Keenan]{Liam Keenan}
\author[J.C. Moreno]{Juan C. Moreno}
\author[J.D. Quigley]{J.D. Quigley}

\begin{document}

\maketitle

\begin{abstract}
We develop homological Real trace methods, an approach to understanding the continuous mod two Bredon homology of Real topological periodic homology and Real topological negative cyclic homology, generalizing prior work of Bruner and Rognes. We use this new approach to do several computations, including the continuous mod two Bredon homology of the Real topological negative cyclic homology of Real bordism. 
\end{abstract}

\section{Introduction}
In recent years, the study of algebraic $K$-theory has seen huge developments via trace methods. In this approach, algebraic $K$-theory is approximated by topological analogues of  constructions from classical algebra, including topological Hochschild homology (THH), topological periodic homology (TP), topological negative cyclic homology (TC$^{-}$), and topological cyclic homology (TC). There is a cyclotomic trace map \cite{BHM93} relating algebraic $K$-theory and topological cyclic homology 
\[
K(R) \to \TC(R),
\] 
and the Dundas--Goodwillie-McCarthy theorem \cite{DGM12} says that algebraic $K$-theory is well-approximated by topological cyclic homology. 
By \cite{NS18}, for any connective $E_1$-ring $R$ there is a fiber sequence (after $p$-completion):
$$\TC(R) \to \TC^-(R) \to \TP(R),$$
so understanding $\TC^{-}$ and $\TP$ gives computational purchase on $\TC$. 

Topological Hochschild homology is a $\mathbb{T}$-equivariant spectrum, where $\mathbb{T}$ denotes the circle group. Topological negative cyclic homology and topological periodic homology are defined as the homotopy $\mathbb{T}$-fixed points and the $\mathbb{T}$-Tate construction of THH respectively. Thus, the key tools for computing $\TC^{-}(R)$ and $\TP(R)$ are the homotopy fixed point and Tate spectral sequences. One challenge of working with these spectral sequences is that they take as input $\pi_\ast(\THH(R))$. In other words, in order to work with them one must have a full understanding of the homotopy groups of $\THH(R)$. 

In response to this difficulty, in \cite{BR05}, Bruner and Rognes developed the theory of ``homological trace methods" to compute the (continuous) mod $p$ homology groups of topological periodic and negative cyclic homology. In particular, they developed homological homotopy fixed point and Tate spectral sequences, which take as input the mod $p$ homology of $\THH(R)$ rather than the homotopy groups. The continuous homology groups computed by Bruner and Rognes' spectral sequences can then be used as input for the inverse limit Adams spectral sequence, which in principle can be used to recover the $p$-complete homotopy type of $\TP$ and $\TC^-$. The advantage of working with homology first is that the B{\"o}kstedt spectral sequence gives a good handle on the mod $p$ homology of $\THH$ of many ring spectra, while the homotopy groups of $\THH$ are often much harder to compute. Bruner and Rognes use this approach to study the topological negative cyclic homology of spectra including $MU$ and $BP$. We review some of their work in \cref{Sec:Background}.

Real algebraic K-theory ($\KR$), defined by Hesselholt and Madsen \cite{HM13}, is an invariant of rings with anti-involution which simultaneously encodes algebraic K-theory, Hermitian K-theory, and L-theory. The development of trace methods in the Real setting has been the subject of much research (see, for instance, \cite{Dot12, DMPR17, QS21b, AKGH21, DMP24, Par24, AKKQ25, yang2025filteredhochschildkostantrosenbergtheoremreal}). As with ordinary algebraic K-theory, Real algebraic K-theory is approximated by a series of Real trace invariants called Real topological Hochschild homology ($\THR$), Real topological cyclic homology ($\TCR$), Real topological periodic homology ($\TPR$), and Real negative cyclic homology ($\TCR^-$). Just as the Dundas--Goodwillie--McCarthy theorem implies that algebraic K-theory is well-approximated by topological cyclic homology, Harpaz, Nikolaus, and Shah have announced results which imply that $\TCR$ is a good approximation to $\KR$. 

In this paper, we develop an analogue of Bruner and Rognes' homological trace methods for Real trace theories. Real topological Hochschild homology (THR) is an $O(2)$-equivariant spectrum, but in contrast with classical trace methods where $\TC^-$ and $\TP$ are obtained by taking the homotopy fixed points and Tate construction with respect to $\t$, the Real analogues $\TCR^-$ and $\TPR$ are defined using the \emph{parametrized} homotopy fixed point and Tate constructions, which take as input $C_2$-spectra with twisted $\t$-action \cite{QS25}. These more complicated constructions are necessary to ensure that $\TCR^-$ and $\TPR$ are genuine $C_2$-spectra. See \cref{Sec:SS} for further discussion. 

For a $C_2$-spectrum with twisted $\t$-action $X$, the $C_2$-spectra $X^{h_{C_2}\t}$ and $X^{t_{C_2}\t}$ can be analyzed using the parametrized homotopy fixed point and Tate spectral sequences, whose input is $\pi_\star X$, the $RO(C_2)$-graded homotopy groups of $X$. As these groups are often hard to compute, we are motivated to explore an alternative based on Bredon homology. In particular, we construct a homological parametrized homotopy fixed point spectral sequence. A homological parametrized Tate spectral sequence can be constructed analogously. Below, we write $\mathcal{A}_\star^{C_2}$ for the $C_2$-equivariant dual Steenrod algebra \cite{HK01}.
We write $H_\star = H_\star(pt;\mft)$ for the $RO(C_2)$-graded Bredon homology of a point with coefficients in the constant $C_2$-Mackey functor $\mft$.

\begin{ThmAlpha}[{\cref{thm:SSadd} and \cref{Prop:MultStr}}]
Let $X$ be a $C_2$-spectrum with twisted $\t$-action. There is a conditionally convergent spectral sequence of $\mathcal{A}_\star^{C_2}$-comodules 
\[E^2_{\ast,\star} = H_\star(X)\langle\bar{y}\rangle \Rightarrow H^c_\star(X^{h_{C_2}\t}; \mft )\]
converging to the continuous $\mft$-Bredon homology of $X^{h_{C_2}\t}$.

If $X$ is moreover an $E_{\infty}$-algebra in $C_2$-spectra with twisted $\t$-action, then this is a multiplicative spectral sequence of $\mathcal{A}_\star^{C_2}$-comodule algebras, where $H_\star(X)$ has the Pontryagin product and the angle brackets indicate that there is a multiplicative relation 
\[\bar{y}a = \mu_\star(a)\bar{y}\]
where $\mu_\star $ is induced by the restriction of the $\t$-action map to the Weyl group $\mu_2$ of $C_2$ in $O(2)$ (see \cref{SS:d2mu} for details).
\end{ThmAlpha}
The continuous Bredon homology appearing in the statement will be defined in \cref{def:continuous-bredon-homology}.

We describe the differentials in this spectral sequence in \cref{Sec:Diff}, lifting results of Bruner and Rognes to this context. Bruner and Rognes also produced many infinite cycles in their spectral sequence using Dyer--Lashof operations. In classical stable homotopy theory, these Dyer--Lashof operations arise in the mod $p$ homology of any $\mathbb{E}_\infty$-ring spectrum, but in $C_2$-equivariant homotopy theory, one obtains equivariant Dyer--Lashof operations only when working with $C_2$-$\mathbb{E}_\infty$-ring spectra, which typically have more structure than $\mathbb{E}_\infty$-rings in $C_2$-spectra; we refer to \cite{Yan25} and \cite{LLP25} for further discussion on equivariant multiplicative structures. While a full study of $C_2$-equivariant Dyer--Lashof operations has not been made, we are able to use partial results of D. Wilson \cite{Wil17} to prove the following:

\begin{ThmAlpha}[{\cref{thm:pc}}]
Let $R$ be a $C_2$-$\mathbb{E}_\infty$-algebra in the category of $C_2$-spectra with twisted $\t$-action. 
Suppose that $x \in H_V \left(R;\mft \right)$ survives to the $E^{2r}$-term $E^{2r}_{0,V} \subset H_V \left(R;\mft \right)$ of the homological parametrized homotopy fixed point spectral sequence for $R$, and there is a nontrivial differential $d^{2r}(x) = \bar{y}^r \cdot \delta x$.
\begin{enumerate}[(a)]
	\item If $|x| = k\rho$, then the $2r$ classes
	\[Q^{k\rho}(x), \ Q^{k\rho+\sigma}(x), Q^{(k+1)\rho}(x), \ \ldots, \ Q^{(k+r-1)\rho}(x), \ \text{ and } \ Q^{(k+r-1) \rho+\sigma}(x) + x \delta x\]
	all survive to the $E^\infty$-term.
	\item If $|x| = k\rho+1$, then the $2r$ classes
	\[Q^{k\rho+\sigma}(x), Q^{(k+1)\rho}(x), Q^{(k+1)\rho+\sigma}(x), \ \ldots, \ Q^{(k+r-1)\rho + \sigma}(x), \ \text{ and } \ Q^{(k+r)\rho}(x) + x \delta x\]
	all survive to the $E^\infty$-term.
\end{enumerate}
\end{ThmAlpha}

As an application, we have the following computations:  

\begin{ThmAlpha}[{\cref{thm:HcTHRMUR}}]
The $E^\infty$-term of the homological parametrized homotopy fixed point spectral sequence converging to the continuous homology, $H^c_{\star} \left(\TCR^-(\MUR); \mft \right)$, is given by
\[E^\infty_{\ast,\star} \cong H_\star \left[ \bar{y}, \bar{b}_i^2 \mid  i>0 \right] \otimes_{H_\star} \Lambda_{H_\star} \left[ \bar{b}_i\bar{\sigma}\bar{b}_i \mid  i>0 \right] \]
plus a simple $\bar{y}$-torsion module in filtration zero.
\end{ThmAlpha}

\begin{ThmAlpha}[{\cref{thm:HcTHRBPRn}}]
Suppose $\BPR\langle n\rangle$ is a $C_2$-$E_\infty$-ring.
Then the $E^\infty$-term of the homological parametrized homotopy fixed point spectral sequence converging to the continuous homology
$H_\star^c\left(\TCR^{-}(\BPR\langle n\rangle) ; \mft \right)$ is given by 
\[E^\infty_{\ast,\star}\cong \frac{H_\star \left[\bar{\xi}_i',\bar{\tau}_i' \mid i\geq n+2 \right]}{(\bar{\tau}_j')^2 = u_\sigma\bar{\xi}_{j+1}'+a_\sigma\bar{\tau}_{j+1}'}\otimes_{H_\star}H_\star \left[\bar{\xi}_1^2,...,\bar{\xi}_{n+1}^2,\bar{y}\right] \otimes_{H_\star}\Lambda_{H_\star} \left[\bar{\xi}_1\bar{\sigma}\bar{\xi}_1,...,\bar{\xi}_{n+1}\bar{\sigma}\bar{\xi}_{n+1} \right]\] plus a simple $\bar{y}$-torsion module in filtration zero.
\end{ThmAlpha}

The spectral sequence that we construct here is obtained from an $O(2)$-equivariant filtration on 
the parametrized universal bundle, $E_{C_2}\t$. This filtration has the advantage of restricting to the same one used by Bruner--Rognes, allowing us to compare our computations to theirs. 
Alternatively, one could apply parametrized homotopy fixed points to the $C_2$-slice filtration, and consider the associated spectral sequence.
This filtration will be more harmonious with equivariant multiplicative structures, and we plan to address this in future work.

An important aspect of $C_2$-spectra with twisted $\t$-actions is that their geometric fixed point spectrum inherits an action of the Weyl group $\mu_2 = \{\pm 1\}\subset\t$ of $C_2$ in $O(2)$. 
Along the way we also consider a homological spectral sequence for the $\mu_2$-homotopy fixed points of the geometric fixed points of real trace invariants. 
This spectral sequence receives a comparison map from the homological parametrized homotopy fixed point spectral sequence. Pushing forward permanent cycles along this map allows us to 
deduce the following. 

\begin{CorAlpha}[{\cref{cor:HcgfpTHRbordism}}]
	The $E^\infty$-term of the homological homotopy fixed point spectral sequence converging to the continuous homology, $H_\ast^c \left( (\THR(\MUR)^{\Phi C_2})^{h\mu_2} ; \mathbb{F}_2 \right)$, is given by
	\[\F_2 \left[z, \overline{\Phi}(\bar{b}_i^2) \mid  i>0 \right]\otimes\Lambda_{\F_2} \left[\overline{\Phi}(\bar{b}_i\bar{\sigma}\bar{b}_i) \mid i>0 \right]\]
	plus a simple $z$-torsion module in filtration zero. The analogous result for $\BPR$ also holds. 
\end{CorAlpha}

\begin{CorAlpha}[{\cref{cor:HcgfpBPRn}}]
For $-1 \leq n \leq 2$, the $E^\infty$-term of the homological homotopy fixed point spectral sequence converging to the continuous homology
$H^c_\ast  \left( \left( \THR(\BPR\langle n\rangle)^{\Phi C_2} \right)^{h\mu_2}; \mathbb{F}_2 \right)$, is given by 
\[\frac{H_\ast ( \mathrm{BP}\langle n\rangle ;\mathbb{F}_2) \left[\bar{\zeta}_1',...,\bar{\zeta}_{n+1}' \right][z,b_n']}{((\bar{\zeta}_i')^2-\bar{\zeta}_i^4)},\]
plus a simple $z$-torsion module in filtration zero.
\end{CorAlpha}

\subsection{Organization}

In \cref{Sec:Background}, we recall Bruner and Rognes' work in the classical setting and some background on Real trace methods. In \cref{Sec:SS}, we construct the homological parametrized homotopy fixed point spectral sequence and a related spectral sequence for geometric fixed points. In \cref{Sec:Diff}, we analyze the differentials in the spectral sequence, giving a useful formula in the case that our input is Real topological Hochschild homology. In \cref{Sec:Mult} we discuss multiplicative structure. In \cref{Sec:Inf}, we produce families of permanent cycles in the spectral sequence provided that our input is a $C_2$-$\mathbb{E}_\infty$-algebra. In \cref{Sec:Compute}, we apply our results to compute the continuous mod two homology of the Real topological negative cyclic homology of Real cobordism and Real Brown--Peterson homology.

\subsection{Conventions}

We adopt the following notation:
\begin{enumerate}
\item the cyclic group of order two is denoted $C_2$, and the circle group is denoted $\t$;
\item the Eilenberg--MacLane $C_2$-spectrum associated to the constant Mackey functor on $\mft$ is denoted $H$;
\item we write $H_\star = H_\star(pt;\mft) \cong \pi^{C_2}_\star H\mft$ for the $RO(C_2)$-graded Bredon homology of a point, and $H_\ast = H_\ast(pt;\mathbb{F}_2)\cong \pi_\ast H\mathbb{F}_2$ for the nonequivariant homology of a point;
\item the representation ring of real orthogonal representations of $C_2$ is denoted $RO(C_2)$;
\item the trivial representation of $C_2$ is denoted $1$, the sign representation is denoted $\sigma$, and the regular representation is denoted $\rho$;
\item given a $C_2$-representation $V$, we write $S(V)$ for the unit sphere in $V$ and $S^V$ for its one-point compactification.
\end{enumerate}

\subsection{Acknowledgements} 

We thank the organizers of the Collaborative Research Workshop on K-theory and Scissors Congruence, the Vanderbilt University Mathematics Department, the University of Pennsylvania Mathematics Department, and the NSF for supporting this workshop as part of the focused research collaboration grant (FRG) DMS-2052977. The authors thank Gabriel Angelini-Knoll and Mike Hill for helpful discussions. MC was partially supported by NSF grants DMS-2405030, DMS-2052846, and DMS-2104348. TMG was supported by NSF grants DMS-2404932, DMS-2104233, and DMS-2052042. JDQ was supported by NSF grants DMS-2441241 and DMS-2414922.

\section{Background}\label{Sec:Background}
\subsection{Summary of the work of Bruner and Rognes}
In \cite{BR05}, Bruner and Rognes pioneered the use of \emph{homological trace methods} to understand $\TC^-(R)$ and $\TP(R)$, especially in cases where $\pi_*\THH(R)$ is not well-understood but $H_*(\THH(R);\mathbb{F}_p)$ is. Following their lead, we will focus on $\TC^-(R)$ in this summary. 

Let $X$ be a spectrum with $\t$-action.
There are equivalences
\[   X^{h\t} \simeq F(E\t_+, X)^{\t} \simeq F(S(\c^\infty)_+, X)^{\t}, \]
where $\t$ acts on $\c$ by rotation. Applied to $X = \THH(R)$, this gives
\[ \TC^-(R) \coloneq \THH(R)^{h\t} \simeq F(S(\c^\infty)_+, \THH(R))^{\t}. \]
The $\t$-equivariant CW-filtration of $S(\c^\infty)$ by $S(\c^r)$, $r \geq 0$, gives rise to a filtration of $X^{h\t}$; applying mod $p$ homology gives rise to a spectral sequence
\[ E^2_{s,t} \cong H^{-s}_{gp}(\mathbb{T}; H_t(X;\mathbb{F}_p)) \Rightarrow H_{s+t}^c(X^{h\mathbb{T}};\mathbb{F}_p) \]
starting with the group cohomology of $\t$ with coefficients in the homology of $X$ converging to the \emph{continuous} homology of $X^{h\mathbb{T}}$, 
\[  H_\ast^c(X^{h\t};\mathbb{F}_p) \coloneq \lim_r H_\ast\bigl( F(S(\c^r)_+,X)^{\t};\mathbb{F}_p\bigr).\] We call this the \emph{homological homotopy fixed point spectral sequence} for 
$X$ and denote it by $\textrm{HHFPSS}(X)$.

When $X = \THH(R)$, this spectral sequence gives a way of passing from $H_*(\THH(R);\mathbb{F}_p)$ to $H^c_*\TC^-(R;\mathbb{F}_p)$. We note that $H_\ast(\THH(R);\mathbb{F}_p)$ is often accessible using the B\"okstedt spectral sequence, 
allowing one to compute the $E^2$-page for $\textrm{HHFPSS}(\THH(R))$. 

One of the key takeaways from the work of Bruner and Rognes is that we also have very good control of the differentials in this spectral sequence. 
Their first main theorem, \cite[Thm. 1.1]{BR05}, provides several important properties of this spectral sequence. In particular, they obtain an explicit description of the $E_2$-term, explain how differentials in the spectral sequence interact with Dyer--Lashof operations, and describe certain consequences of multiplicative structure. 

Using these properties, Bruner and Rognes in \cite[Thm. 1.2]{BR05} go on to identify certain infinite cycles in the spectral sequence. These infinite cycles give rise to many interesting classes in $H^c_*(\TC^-(R);\mathbb{F}_p)$, and allow Bruner and Rognes to determine the $E^\infty$-pages for $R = BP\langle n \rangle, ko, tmf, \mathrm{MU}$, and $\mathrm{BP}$.

\subsection{$C_2$-spectra with twisted $\t$-action}
As the Hochschild-type invariants appearing in Real trace methods come naturally equipped with a $C_{2}$-twisted $\TT$-action, we must appeal to parameterized homotopy theory; specifically, the parameterized homotopy fixed point and Tate constructions. 
In this section, we recall some of the basic theory developed by the fifth author and Shah in \cite{QS25}. We begin by recalling the basics of $C_2$-\categories.

Recall that for any (topological) group $G$, the \textit{orbit category} $\mathcal{O}_G$ is the full subcategory of $\Top^{BG}$ spanned by objects equivalent to $G/H$ for $H$ a closed subgroup of $G$.

\begin{definition}\label[definition]{C2catspcdef}
A \emph{$C_{2}$-\category} is a coCartesian fibration $\mathcal{C}\rightarrow \mathcal{O}^{\op}_{C_{2}}$. 
A \emph{$C_{2}$-functor} between $C_{2}$-\categories $p\colon \mathcal{C} \rightarrow \mathcal{O}^{\op}_{C_{2}}$ and $q\colon \mathcal{D}\rightarrow \mathcal{O}^{\op}_{C_{2}}$ is a commutative diagram 
\[\begin{tikzcd}
	{\mathcal{C}} && {\mathcal{D}} \\
	& {\mathcal{O}_{C_{2}}^{\op}}
	\arrow["F", from=1-1, to=1-3]
	\arrow["p"', from=1-1, to=2-2]
	\arrow["q", from=1-3, to=2-2]
\end{tikzcd}\]
where $F$ sends $p$-coCartesian arrows to $q$-coCartesian arrows. 
We denote the \category of $C_{2}$-functors by $\Fun_{C_{2}}(\mathcal{C},\mathcal{D})$. 
\end{definition}

\begin{example}
	Given a pair of $C_2$-\categories, $\mathcal{C},\mathcal{D}$, there is a $C_2$-\category 
	$\underline{\Fun}(\mathcal{C},\mathcal{D})$ with fiber over the fixed orbit given by $\Fun_{C_2}(\mathcal{C},\mathcal{D})$. 
\end{example}

\begin{example}
	There are $C_2$-\categories of $C_2$-spaces, $\underline{\Top}^{C_2}$, and $C_2$-spectra, $\underline{\Sp}^{C_2}$, which recover 
	the \categories of $C_2$-spaces, $\Top^{C_2}$, and $C_2$-spectra, $\Sp^{C_2}$, as the respective fibers over the fixed orbit.
	The fibers over the free orbit are the \categories $\Top$ and $\Sp$, respectively.  
\end{example}

\begin{example}\label[example]{example: finite pointed C_2 sets}
	Let $\Fin_{C_2}$ denote the category of finite $C_2$-sets. There is a $C_2$-\category 
	\[
	\underline{\Fin}_{C_2,\ast} \rightarrow \mathcal{O}^{\op}_{C_{2}}
	\]
	whose fiber over $C_{2}/e$ is $(\Fin_{C_2})_{C_{2}//C_{2}}$ and whose fiber over $C_{2}/C_{2}$ is $(\Fin_{C_2})_{\ast//\ast}$; see \cite[Section 2.1]{nardin2022parametrizedequivarianthigheralgebra} for further details. 
\end{example}

Let $\t \subseteq \c^\times$ be the circle group equipped with $C_2$-action by complex conjugation, which furnishes an extension of groups 
$$1 \to \t \to O(2) \to C_2 \to 1.$$
We fix a section of $O(2)\to C_2$ with generator $\sigma\in O(2)$. The Weyl group of $C_{2} = \langle \sigma \rangle \subseteq O(2)$ is given by $\mu_{2} \subseteq \t$, the cyclic subgroup of order two inside $\t$. 

\begin{definition}
	Let $B_{C_2}\t \subseteq \mathcal{O}_{O(2)}^{\op}$ denote the full subcategory spanned by the $O(2)$-orbits which are $\t$-free.
	We view $B_{C_2}\t$ as a $C_2$-\category via the composition 
	\[\rho_{\t}\colon B_{C_2}\t \hookrightarrow \mathcal{O}_{O(2)}^{\op} \xrightarrow{(-)/\t} \mathcal{O}_{C_2}^{\op}.\]
	It follows from \cite[Lemma 3.16]{QS25} that $\rho_\t$ is a left fibration. 
\end{definition}

\begin{remark}
The reader may benefit from the following visual representation of the left fibration $\rho_{\t} \colon B_{C_{2}}\t \rightarrow \mathcal{O}_{C_{2}}^{\op}$, where the labelling of the arrows corresponds to the space of maps between the indicated objects:
\[\begin{tikzcd}
	{O(2)/e} && {O(2)/C_{2}} \\
	{C_{2}/e} && {C_{2}/C_{2}}
	\arrow["{O(2)}", from=1-1, to=1-1, loop, in=100, out=170, distance=10mm]
	\arrow[""{name=0, anchor=center, inner sep=0}, "{O(2)/C_{2}}", tail reversed, no head, from=1-1, to=1-3]
	\arrow["\mu_2", from=1-3, to=1-3, loop, in=10, out=80, distance=10mm]
	\arrow["{C_{2}}", from=2-1, to=2-1, loop, in=190, out=260, distance=10mm]
	\arrow[""{name=1, anchor=center, inner sep=0}, "\ast" swap, tail reversed, no head, from=2-1, to=2-3]
	\arrow["\ast", tail reversed, no head, from=2-3, to=2-3, loop, in=280, out=350, distance=10mm]
	\arrow["{\rho_{\t}}", between={0.2}{0.8}, squiggly, from=0, to=1]
\end{tikzcd}\]
\end{remark}

\begin{definition}
A \emph{${C_2}$-spectrum with twisted $\t$-action} is a ${C_2}$-functor $X \colon B_{C_2} \t \to \Sp^{C_2}$ from $B_{C_2}\t$ to the ${C_2}$-category of genuine ${C_2}$-spectra.
The \emph{category of $C_2$-spectra with twisted $\t$-action} is the $C_2$-functor category \[\Sp^{h_{C_2}\t} \coloneq \Fun_{C_2}(B_{C_2}\t,\underline{\Sp}^{C_2}).\]
\end{definition}

\begin{remark}
Informally, a $C_{2}$-spectrum with twisted $\t$-action is an $O(2)$-spectrum which is genuine with respect to the subgroup $C_{2}$. 
Restricting along the evident inclusion $BO(2)\subset B_{C_2}\t$ provides a functor $\Sp^{h_{C_2}\t}\to \Sp^{BO(2)}$, which one can think of as a $C_2$-equivariant Borel completion. 
Furthermore, by recollement theory, one can recover the original twisted $C_2$-spectrum from its Borel completion via a pullback diagram: 
\[
\begin{tikzcd}
\Sp^{h_{C_{2}}\t} \arrow[r] \arrow[d] & \Fun([1],\Sp^{B\mu_{2}}) \arrow[d,"\mathrm{ev}_{1}"] \\
\Sp^{BO(2)} \arrow[r,"(-)^{tC_{2}}"'] & \Sp^{B\mu_{2}} 
\end{tikzcd}
\]
Thus, the data of a $C_{2}$-spectrum with twisted $\t$-action consists of a pair $(X,Y \rightarrow X^{tC_{2}})$  where $X$ is a spectrum with $O(2)$-action, $Y$ is a spectrum with $\mu_{2}$-action, and $Y \rightarrow X^{tC_{2}}$ is a  $\mu_{2}$-equivariant map. 
\end{remark}

\begin{definition}
Given a ${C_2}$-spectrum with twisted $\t$-action, $X\in\Sp^{h_{C_2}\t}$, we define the \emph{parametrized homotopy orbits} and \emph{parametrized homotopy fixed points} 
as the $C_2$-colimit and $C_2$-limit
\[X_{h_{C_2}\t} \coloneq \colim_{B_{C_2}\t} X, \quad X^{h_{C_2}\t} \coloneq \lim_{B_{C_2}\t} X,\] which are objects in $\Sp^{C_2}$. 
The \emph{parametrized Tate construction} is then defined as the cofiber \[X^{t_{C_2}\t} \coloneq \cofib(\mathrm{Nm}\colon \Sigma^{\sigma} X_{h_{C_2}\t} \to X^{h_{C_2}\t}),\]
where $\mathrm{Nm}$ is the norm constructed in \cite[Theorem D]{QS25}.
\end{definition}

\begin{remark} 
We can obtain geometric models of these constructions following \cite[Remark 5.51]{QS25}. 
Explicitly, we have equivalences of ${C_2}$-spectra
$$X_{h_{C_2}\t} \simeq (E_{C_2}\t_+ \wedge X)^\t, \quad X^{h_{C_2}\t} \simeq F(E_{C_2}\t_+, X)^\t, \quad X^{t_{C_2}\t} \simeq (\widetilde{E_{C_2}\t} \wedge F(E_{C_2}\t_+,X))^\t,$$
where $E_{C_2}\t$ is the universal ${C_2}$-space for ${C_2}$-equivariant principal $\t$-bundles, and $\widetilde{E_{C_2}\t}$ is the cofiber of 
the collapse map $E_{C_2}\t_+\to S^0$. 
\end{remark}

\subsection{Real trace invariants} As in the nonequivariant case, Hesselholt and Madsen's Real algebraic K-theory \cite{HM13} is approximated by a series of Real trace invariants, each built from structure possessed by Real topological Hochschild homology, whose definition we now recall.

\begin{definition}[{\cite[Definition~5.2]{QS21b}}]
Let $A$ be a ${C_2}$-$\mathbb{E}_\infty$-ring. The \emph{Real topological Hochschild homology} of $A$, \[\THR(A) \colon = S^\sigma \otimes A,\] is the $C_2$-tensor of $A$ with the $C_2$-space $S^\sigma$ in the $C_2$-$\infty$-category of $C_2$-$\mathbb{E}_\infty$-rings in $\underline{\Sp}^{C_2}$. 
\end{definition}
The circle action on $S^\sigma$ gives $\THR(A)$ the structure of a ${C_2}$-spectrum with twisted $\t$-action. 
From a suitable cell structure on $S^\sigma$, we can deduce that there is an equivalence \cite[Remark~5.3]{QS21b} 
\[\THR(A)\simeq A\underset{N_e^{C_2}i^\ast_e A}{\wedge}A,\] where we view $A$ as a module over the norm via the counit map $\varepsilon\colon N_e^{C_2}i^\ast_e A\to A$ of the norm-restriction adjunction (see also \cite{DMPR17}). 
More generally, if $A$ is an $\mathbb{E}_\sigma$-ring, then $A$ can be given a left and right module structure over the norm, and one can take this smash product 
as the definition of Real topological Hochschild homology. 

From this definition we can understand the geometric fixed points of Real topological Hochschild homology together with the residual $\mu_2$-action \cite[Theorem 2.26]{DMPR17}. 

\begin{proposition}\label[proposition]{cor:residualmugfp}
There is a $\mu_2$-equivariant equivalence \[\THR(A)^{\Phi C_2}\simeq A^{\Phi C_2} \underset{i^\ast_e A}{\wedge} A^{\Phi C_2}\] where $\mu_2$ acts on $A^{\Phi C_2}\underset{i^\ast_e A}{\wedge} A^{\Phi C_2}$ by swapping the two factors. 
\end{proposition}

We can use the twisted $\t$-action to define \emph{Real topological negative cyclic homology} and \emph{Real topological periodic cyclic homology},
$$\TCR^-(A) \colon = \THR(A)^{h_{C_2}\t}, \quad \TPR(A) \colon = \THR(A)^{t_{C_2}\t}.$$

A discussion of Real topological cyclic homology is not necessary for objectives in this paper, but see \cite[Definition~2.10]{QS21b} for a definition, where it is also shown that these invariants fit into a fiber sequence of ${C_2}$-spectra (after suitable completion)
$$\TCR(A) \to \TCR^-(A) \to \TPR(A),$$
so we may study $\TCR(A)$ by understanding $\TCR^-(A)$ and $\TPR(A)$. Thus, developing computational tools for $\TCR^-(A)$ and $\TPR(A)$ is key to understanding Real topological cyclic homology and Real algebraic $K$-theory. 
Inspired by Bruner--Rognes, we will focus on studying these trace invariants through $C_2$-equivariant homology.

\subsection{$C_2$-equivariant homology}
Let $H_\star = \pi_\star H\mft$ be the $C_2$-equivariant homotopy group of the Eilenberg--Maclane spectrum for the constant $C_2$-Mackey functor $\mft$. The $H$-homology of a fixed point was first computed in unpublished work by Stong. We refer the reader to \cite{HK01} for a detailed account. The answer is a square-zero extension of a polynomial algebra: 
\[H_\star \cong \F_2[u_\sigma,a_\sigma]\oplus \F_2\left\{\frac{\theta}{u_\sigma^k a_\sigma^\ell} \mid k,\ell\geq 0\right\},\]
where $|u_\sigma| = 1-\sigma, |a_\sigma| = -\sigma$, and $\theta$ is an infinitely $a_\sigma,u_\sigma$-divisible $a_\sigma,u_\sigma$-torsion class in degree $2\sigma-2$. 
The corresponding $C_2$-equivariant Steenrod algebra was first computed by Hu--Kriz in
\cite{HK01} (see also \cite{LSWX19} for a nice exposition). We have 
\[\mathcal{A}_\star^{C_2} = (H \mft)_\star (H\mft) \cong \frac{H_\star[\bar{\xi}_{i+1},\bar{\tau}_i \mid i\geq 0]}{\bar{\tau}_i^2 = u_\sigma\bar{\xi}_{i+1}+a_\sigma\bar{\tau}_{i+1}},\] where 
$|\bar{\xi}_i| = (2^i-1)\rho$ and $|\bar{\tau}_i| = (2^i-1)\rho+1$. The bars here indicate that we use the conjugates of the Hu--Kriz generators. 

The geometric fixed points of $H$ are 
\[(H\mft)^{\Phi C_2}\simeq (a_\sigma^{-1} H\mft)^{C_2}\] 
with $\pi_{\star}$ given by $H_{\star}[b]$, where $b$ is represented by $u_\sigma a_\sigma^{-1} \in a_\sigma^{-1}H_\star$.
Composition with the quotient $H{\mathbb{F}}_2[b]\to H{\mathbb{F}}_2$ then defines a \emph{(reduced) geometric fixed point} map 
\[\overline{\Phi}\colon H_V (X;\mft)\to H_{|V^{C_2}|} (X^{\Phi C_2};\mathbb{F}_2),\] 
which for any $X\in\Sp^{C_2}$ and $V\in RO(C_2)$ sends 
\[\overline{\Phi}\colon (S^V\xrightarrow{x}H\mft \wedge X)\mapsto (S^{V^{C_2}}\xrightarrow{\Phi^{C_2}(x)}H{\mathbb{F}}_2 [b]\wedge X^{\Phi C_2}\to H{\mathbb{F}}_2\wedge X^{\Phi C_2}).\]
The following lemma is very useful for computing homology of $C_2$-spectra and follows simply from the 
fact that equivalences are detected on underlying and geometric fixed points.

\begin{lemma}[{\cite[Lemma~2.8]{BehrensWilson}}]\label[lemma]{lemma:BWLemma}
Let $\{x_i\}\subset H_\star (X;\mft)$ be a collection of elements. If $\{\res(x_i)\}$ is an $\F_2$-basis for $H_\ast (i^\ast_e X; \mathbb{F}_2)$ and $\left\{\overline{\Phi}(x_i)\right\}$ is an $\F_2$-basis for $H_\ast (X^{\Phi C_2}; \mathbb{F}_2)$, then $\{x_i\}$ is an $H_\star$-basis for $H_\star (X; \mft)$. 
\end{lemma}

If $R$ is $C_2$-$\mathbb{E}_\infty$-ring, then we have additional structure on $H_\star (R; \mft)$ that will be relevant. We have the \emph{homology norm map} 
\[n\colon H _t ( i^\ast_e R;\mathbb{F}_2) \to H_{t\rho} (R;\mft)\] 
defined as follows.

Under the identification $i_e^\ast H\mft\simeq H\mathbb{F}_2$, we represent a class $x\in H_t(i_e^\ast R;\mathbb{F}_2)$ by a map $x\colon S^t\to i_e^\ast H\mft \wedge i_e^\ast R$, and set
\[
\bigl(S^t \xrightarrow{x}     i_e^\ast H\mft \wedge i_e^\ast R\bigr)   \longmapsto   \bigl(S^{t\rho} \xrightarrow{N(x)} N(i_e^\ast H\mft )\wedge N(i_e^\ast R)     \xrightarrow{\varepsilon_H\wedge\varepsilon_R}     H\mft \wedge R\bigr),
\]
where $\varepsilon_H\colon N(i_e^\ast H\mft ) \to H\mft $ denotes the norm structure map of $H\mft $.
The following is then readily verified from the definitions and is useful for studying the homology of the geometric fixed points of $\THR$. 

\begin{proposition}\label[proposition]{prop:gfpmodulestructure}
Let $R$ be a $C_2$-$\mathbb{E}_\infty$-ring. Then the map induced on homology by $\Phi^{C_2}(\varepsilon_R)\colon i^\ast_e R\to R^{\Phi C_2}$ is given by 
\[\overline{\Phi}\circ n\colon  H_\ast (i^\ast_e R; \mathbb{F}_2)  \to H_\ast (R^{\Phi C_2} ; \mathbb{F}_2).\]  
\end{proposition}

There are also equivariant power operations on the homology of such $R$ \cite{Wil17,Wil19}
\[Q^{k\rho+\epsilon\sigma}\colon H_V (R;\mft)\to H_{V+k\rho+\epsilon\sigma}(R;\mft),\quad k \geq 0, \ \epsilon\in\{0,1\}.\]
These satisfy the usual properties of power operations, for example if $x\in H_V (R;\mft)$, for some $V$ of the form $k\rho+\epsilon\sigma$, then $Q^V x = x^2$.  
They also restrict to the Dyer--Lashof operations on underlying homology: $\res(Q^{k\rho+\epsilon\sigma}x) = Q^{2k+\epsilon}\res(x)$. 

\begin{thm}[{\cite[Corollary~1.6.4]{Wil19}}]\label{thm:SteenrodPowerOps}
The action of the equivariant power operations on $\mathcal{A}_\star^{C_2}$ satisfies the following 
\begin{align*}
Q^{2^k \rho}\bar{\tau}_k = \bar{\tau}_{k+1}, \quad
Q^{(2^k -1)\rho +\sigma}\bar{\tau}_k = \bar{\xi}_{k+1}, \quad
Q^{2^k \rho}\bar{\xi}_k = \bar{\xi}_{k+1}.
\end{align*}
\end{thm}

\subsection{Parametrized internal homs}
Let $G$ be a finite group. The goal of this subsection is to show that the internal hom of a closed $G$-symmetric monoidal \category is canonically $G$-lax symmetric monoidal. This result is critical for our work, and, while certainly well-known, does not appear to be documented anywhere in the literature. 

Let $\underline{\rm Fin}_{G,\ast}$ denote the $G$-\category of finite pointed $G$-sets (compare with Example~\ref{example: finite pointed C_2 sets}).

\begin{definition} 
A \textit{$G$-symmetric monoidal \category} is a coCartesian fibration 
\[
\underline{\cal{C}}^{\otimes} \rightarrow \underline{\rm Fin}_{G,\ast}
\]
with the additional property that the canonical functor 
\[
\underline{\cal{C}}^{\otimes}_{U} \rightarrow \prod_{W \in \mathrm{Orbit}(U)} \underline{\cal{C}}^{\otimes}_{W} 
\]
is an equivalence of \categories, where $U \in \underline{\rm Fin}_{G,\ast}$ (we have slightly abused notation by omitting the orbit of $\cal{O}^{\op}_{G}$ over which $U$ lies). Furthermore, a \textit{$G$-symmetric monoidal (resp. lax symmetric monoidal) functor} is a functor of \categories 
\[
\underline{F}^{\otimes} \colon \underline{\cal{C}}^{\otimes} \rightarrow \underline{\cal{D}}^{\otimes}\]
over $\underline{\rm Fin}_{G,\ast}$ which sends coCartesian (resp. inert) arrows to coCartesian (resp. inert) arrows. 
\end{definition}

\begin{example}
For $\underline{\cal{C}}$ a $G$-symmetric monoidal \category, the monoidal product determines a $G$-symmetric monoidal functor of the form 
\[
\underline{\cal{C}}^{\otimes}\times_{\underline{\rm Fin}_{G,\ast}} \underline{\cal{C}}^{\otimes} \rightarrow \underline{\cal{C}}^{\otimes}. 
\]
\end{example}

\begin{definition}
Let $p:\underline{\cal{C}}^{\otimes} \rightarrow \underline{\rm Fin}_{G,\ast} $ be a $G$-symmetric monoidal \category. We define $\CAlg^{G}(\cal{C})$ to be the \category of sections of $p$ which carry inert morphisms to $p$-coCartesian morphisms. Additionally, we define $\coCAlg^{G}(\cal{C})$ to be the \category $\CAlg^{G}(\cal{C}^{\rm vop})^{\op}$. 
\end{definition}

By the duality between coCartesian and Cartesian fibrations we could have alternatively encoded a $G$-symmetric monoidal \category via a \textit{Cartesian fibration} 
\[
\underline{\cal{C}}_{\otimes} \rightarrow (\underline{\rm Fin}_{G,\ast})^{\op}
\]
which satisfies an essentially identical Segal condition. In this language, a $G$-symmetric monoidal (resp. lax symmetric monoidal) functor can be encoded as a functor between Cartesian fibrations which preserves the Cartesian (resp. inert) edges. On first blush, this might seem a categorical indulgence. However, this point of view is essential in what follows. To this end, let 
\[
\mu\colon \underline{\cal{C}}_{\otimes}\times_{(\underline{\rm Fin}_{G,\ast})^{\op}} \underline{\cal{C}}_{\otimes} \rightarrow \underline{\cal{C}}_{\otimes}
\]
denote the $G$-symmetric monoidal functor given by the $G$-symmetric monoidal structure on $\underline{\cal{C}}$. 

\begin{proposition}\label[proposition]{proposition: C-internal hom lax symmetric monoidal}
For $\underline{\cal{C}}$ a presentably $G$-symmetric monoidal \category, the $G$-internal hom
\[
\underline{\Hom}_{\cal{C}} \colon  \underline{\cal{C}}^{\op} \times \underline{\cal{C}} \rightarrow \underline{\cal{C}}
\]
admits a canonical $G$-lax symmetric monoidal structure. 
As a consequence, there is an induced functor 
\[
\coCAlg^{G}(\cal{C})^{\op} \times \CAlg^{G}(\cal{C}) \rightarrow \CAlg^{G}(\cal{C}).
\]
\end{proposition}

\begin{proof}
	First, the existence of $\underline{\Hom}_{\cal{C}}$ follows formally from the parametrized adjoint functor theorem \cite[Theorem 6.2.1]{hilmanParametrisedPresentabilityOrbital2024}. Next, consider the commutative diagram
	\[
	\begin{tikzcd}
	{\underline{\cal{C}}_{\otimes}\times_{(\underline{\rm Fin}_{G,\ast})^{\op}} \underline{\cal{C}}_{\otimes} } && {\underline{\cal{C}}_{\otimes}\times_{(\underline{\rm Fin}_{G,\ast})^{\op}} \underline{\cal{C}}_{\otimes} } \\
	& {\underline{\cal{C}}_{\otimes} }
	\arrow["{(\mathrm{pr}_{1},\mu)}", from=1-1, to=1-3]
	\arrow["{\rm pr_{1}}"', from=1-1, to=2-2]
	\arrow["{\rm pr_{1}}", from=1-3, to=2-2]
	\end{tikzcd}
	\] 
	which is a map of Cartesian fibrations over $\underline{\cal{C}}_{\otimes}$ and $(\mathrm{pr}_{1},\mu)$ is a fiberwise left adjoint. Under the equivalence of \cite[Theorem 3.1.11]{haugsengLaxMonoidalAdjunctions2023}, it is not hard to see the diagram above corresponds to a diagram 
	\[
	\begin{tikzcd}
	{(\underline{\mathcal{C}}^{\mathrm{vop}})^{\otimes}\times_{\underline{\rm Fin}_{G,\ast}}\underline{\mathcal{C}}^{\otimes}} && {(\underline{\mathcal{C}}^{\mathrm{vop}})^{\otimes}\times_{\underline{\rm Fin}_{G,\ast}}\underline{\mathcal{C}}^{\otimes}} \\
	& {(\underline{\mathcal{C}}^{\mathrm{vop}})^{\otimes}}
	\arrow["{(\mathrm{pr}_{1},\underline{\Hom}_{\mathcal{C}}^{\otimes})}", from=1-1, to=1-3]
	\arrow["{\mathrm{pr}_{1}}"', from=1-1, to=2-2]
	\arrow["{\mathrm{pr}_{1}}", from=1-3, to=2-2]
	\end{tikzcd}
	\]
	where $\mathrm{pr}_{1}$ is a coCartesian fibration,
	$(\mathrm{pr}_{1},\underline{\Hom}_{\mathcal{C}}^{\otimes})$ sends inert arrows to inert arrows, and $\underline{\Hom}_{\mathcal{C}}^{\otimes}$ has underlying $G$-functor given by $\underline{\Hom}_{\mathcal{C}}$. Unraveling the definitions, this implies $\underline{\Hom}_{\cal{C}}$ has a canonical $G$-lax symmetric monoidal structure as claimed. The claim about algebras and coalgebras is an immediate formal consequence. 
\end{proof}

\begin{proposition}\label[proposition]{proposition: mapping spectra E_oo}
Let $R$ be a $C_{2}$-$\mathbb{E}_\infty$-algebra in $C_{2}$-twisted $\TT$-spectra and let $X$ be a $C_{2}$-twisted $\TT$-space. 
Then, the $C_{2}$-twisted $\TT$-spectrum $F(X_+,R)$ is $C_{2}$-$\mathbb{E}_\infty$-algebra, naturally in both $X$ and $R$. 
\end{proposition}

\begin{proof}
In virtue of Proposition~\ref{proposition: C-internal hom lax symmetric monoidal}, it is enough to show that the suspension spectrum of a $C_{2}$-twisted $\TT$-space determines a $C_2$-$\mathbb{E}_\infty$-coalgebra in $\uSp^{h_{C_2}\TT}$. By the $C_2$-symmetric monoidality of the suspension spectrum functor, it is enough to show that any $C_{2}$-twisted $\TT$-space is a $C_2$-$\mathbb{E}_{\infty}$-coalgebra. By the results of \cite[3.3.4]{nardin2022parametrizedequivarianthigheralgebra}, we can endow the $C_2$-\category $\underline{\rm Spc}^{h_{C_2}\TT}$ with the $C_2$-Cartesian $C_2$-symmetric monoidal structure; this is because $\underline{\rm Spc}^{h_{C_2}\TT}$ is a parametrized presheaf category and $\underline{\rm Spc}^{C_2}$ is given the $C_2$-Cartesian monoidal structure. Therefore, the vertical opposite of $\underline{\rm Spc}^{h_{C_{2}}\TT}$ is possesses the $C_2$-coCartesian symmetric monoidal structure by \cite[2.4.1]{nardin2022parametrizedequivarianthigheralgebra}. By \cite[4.1.12]{yang2025filteredhochschildkostantrosenbergtheoremreal}, there is an equivalence  
\[
\CAlg^{C_2}((\mathrm{Spc}^{h_{C_{2}}\TT})^{\rm vop}) \simeq (\mathrm{Spc}^{h_{C_{2}}\TT})^{\rm vop}
\]
and thus an equivalence 
\[
\coCAlg^{C_2}(\mathrm{Spc}^{h_{C_{2}}\TT}) \simeq \mathrm{Spc}^{h_{C_{2}}\TT}
\]
completing the proof.
\end{proof}

\section{Homological parametrized homotopy fixed point spectral sequence}\label{Sec:SS}

In this section, we construct a Real analogue of the homological homotopy fixed point spectral sequence. A Real analogue of the homological  Tate spectral sequence can be constructed analogously. Our methods resemble those of Bruner and Rognes, with substantial deviations regarding equivariant multiplicative structure. 

\subsection{The spectral sequence, additively}\label{SS:SSadd}

In this section we construct the homological parametrized homotopy fixed point spectral sequence. The remainder of this section proves the following theorem. 

\begin{thm}\label{thm:SSadd}\label{cons:SStower}
Let $X$ be a $C_2$-spectrum with twisted $\t$-action. There is a conditionally convergent spectral sequence of $\mathcal{A}_\star^{C_2}$-comodules 
$$E^2_{\ast,\star} = H_\star(X;\mft)\langle\bar{y}\rangle \Rightarrow H^c_\star(X^{h_{C_2}\t}; \mft)$$
converging to the continuous $\mft$-Bredon homology of $X^{h_{C_2}\t}$. 
\end{thm}

The unit sphere in $\c^\infty=\bigoplus_{n=1}^\infty \c$, denoted $S(\c^\infty)$, together with its diagonal complex conjugation action, is a model for $E_{C_2}\t$, the universal ${C_2}$-space classifying ${C_2}$-equivariant principal $\TT$-bundles. 
Tracking the ${C_2}$-action on the $\t$-CW structure from \cite[Sec. 2]{BR05}, we have that the equivariant $2n$-filtration is the odd $(2n+1)$-sphere $E_{C_2}\t^{(2n)}=E_{{C_2}}\t^{(2n+1)} = S(\c^{n+1})$. This is obtained from the equivariant $(2n-2)$-filtration $E_{{C_2}}\t^{(2n-2)} = S(\c^n)$ by attaching a free $\t$-equivariant $2n$-dimensional $C_2$-slice cell $\t \times D^{n\rho}$ along the ${C_2}$-equivariant group action map
$$\alpha \colon  \t \times \partial D^{n\rho} \to S(\c^n).$$
The attaching map is $O(2) = \t \rtimes {C_2}$-equivariant if we equip $\partial D^{n\rho}$ with the trivial $\t$-action and $S(\c^n)$ with the free $\t$-action. Thus we obtain an $O(2)$-equivariant filtration
\[\varnothing = E_{C_2}\t^{(-1)}\subset E_{C_2}\t^{(0)} = E_{C_2}\t^{(1)}\subset E_{C_2}\t^{(2)} = E_{C_2}\t^{(3)}\subset E_{C_2}\t^{(4)} = \cdots\]
whose colimit is $E_{C_2}\t$, along with $O(2)$-equivariant cofiber sequences
\begin{equation}\label{eqn:cofiber}
 E_{C_2}\t^{(2n-1)}\to E_{C_2}\t^{(2n)}\to\t_+\wedge S^{n\rho}
\end{equation}
for $n \geq 0$. The action of $\t$ on $S^{n\rho} = D^{n\rho}/\partial D^{n\rho}$ is trivial. 

Applying $F(-,X)^\t$ yields a tower of $C_2$-spectra
\[\cdots \to F(E_{C_2}\t^{(2n)}_+, X)^\t \to F(E_{C_2}\t^{(2n-1)}_+,X)^\t \to \cdots \to F(E_{C_2}\t^{(0)}_+,X)^\t \to *\]
with homotopy limit $X^{h_{C_2}\t}$. The cofiber sequences (\ref{eqn:cofiber}) above induce cofiber sequences of $C_2$-spectra
\[\Sigma^{-n\rho} X \simeq F(\t_+ \wedge S^{n\rho},X)^\t \rightarrow F(E_{C_2}\t^{(2n)}_+,X)^\t \rightarrow F(E_{C_2}\t^{(2n-1)}_+,X)^\t\]
for each $n \geq 0$. Placing $F(E_{C_2}\t^{(-s-1)}_+,X)^\t$ in filtration $s$, we obtain a chain of cofiber sequences of $C_2$-spectra. Applying mod $2$ Bredon homology $H_\star(-;\mft)$ produces an unrolled exact couple as on \cite[pp. 660]{BR05}, giving rise to a spectral sequence computing the limit, continuous Bredon homology.

\begin{definition}\label{def:continuous-bredon-homology}
	For $X\in\Sp^{h_{C_2}\t}$, the \emph{continuous $\mft$-Bredon homology} of $X^{h_{C_2}\t}$ is defined as 
	\[H_\star^c(X^{h_{C_2}\t};\mft)\coloneq\lim_n H_\star(F(S(\c^n)_+,X)^\t;\mft).\]
\end{definition}

We now identify the $E^2$-term in this spectral sequence. For notational convenience, we define the stunted $O(2)$-equivariant space
\[E_{C_2} \t^{m}_a \coloneq E_{C_2}\t^{(m)} / E_{C_2}\t^{(a-1)}, \] and for any $s\in\zz$, let \[v(s)\coloneq\begin{cases}
	s'\rho & \text{if }s = 2s'\\ s'\rho-1 & \text{if }s = 2s'-1.
\end{cases}\]

Define the $E^1$-term as
\begin{equation}\label{ssE1}
\begin{split}
 E^1_{s,V} &\coloneq H_{v(s)+V}\left(F(E_{C_2}\t^{-s}_{-s}, X)^\t; \mft \right) \cong \begin{cases}
H_V(X;\mft) \quad & \text{ if } s =-2n \leq 0, \\
0 \quad & \text{ otherwise.}
\end{cases}
\end{split}
\end{equation}
We write $\bar{y}^n\cdot x \in E^1_{-2n,V}$ for the class that corresponds to $x \in H_V(X;\mft)$ under this identification. Since $E^1$ is concentrated in even filtration degrees, we have $E^{2r-1} = E^{2r}$ for all $r\geq 1$. The spectral sequence has signature
$$E^2_{-2s,V} = \bar{y}^{s}\cdot H_V(X;\mft) \Longrightarrow H^c_{V - s\rho}(X^{h_{C_2}\t}; \mft),$$
and differentials
\[d^{r}\colon E^{r}_{s,V}\to E^{r}_{s-r,V+v(r)-1}.\] We call this the \emph{homological parametrized homotopy fixed point spectral sequence} for $X$ and denote it by 
$\textrm{HHFPSS}_{C_2}(X)$. 

Note that, by construction, the colimit of the tower in \cref{cons:SStower} is contractible. Thus, our spectral sequence is 
conditionally convergent in the sense of \cite[Definition~5.10]{Boa99}. The kernels
\[F_{s,\star} = \ker\left(H_\star^c(X^{h_{C_2}\t};\mft)\to H_\star \left( F(E_{C_2}\t_+^{(-s-1)},X)^\t;\mft \right) \right)\] 
assemble into a complete, exhaustive, increasing filtration of the continuous homology. If additionally Boardman's derived $E_\infty$-term, $RE^\infty$, vanishes, then the spectral sequence will converge strongly, meaning the filtration is also Hausdorff and we have identifications
$E^\infty_{s,\star}\cong F_{s,\star}/F_{s-1,\star}$ \cite[Theorem~7.1]{Boa99}. In particular, strong convergence is guaranteed when the spectral sequence collapses at a finite page. As we will see, this is often the case in examples of interest.

\begin{example}\label[example]{example:sphere}
Consider the case $X = S$, where $S$ is the $C_2$-equivariant sphere spectrum equipped with a trivial twisted $\t$-action. Then $S^{h_{C_2}\t} \simeq F(B_{C_2}\t_+,S) \simeq D \c P^\infty_+$ and $H^c_\star (S^{h_{C_2}\t} ; \mft ) \cong H^{-\star}(\c P^\infty; \mft)\cong H_\star[\bar{y}]$. In this case, our spectral sequence agrees 
with that of \cite[Eq. 2.24]{HK01}. Since $\mft$ is Real oriented, the results of \cite[Sec. 2]{HK01} imply that the spectral sequence collapses at $E^2$. 
\end{example}

\begin{example}\label[example]{example:thr}
If $X = \THR(A)$, then the $\mathrm{HHFPSS}_{C_2}(\THR(A))$ provides a way of computing $H_\star^c(\TCR^{-}(A);\mft)$ from $H_\star(\THR(A);\mft)$. 
\end{example}

We end this section by noting that applying Mackey functor-valued homology to the tower in \cref{cons:SStower} instead gives rise to a spectral sequence of 
$C_2$-Mackey functors. We will not make full use of this structure in this paper, and so we opt to consider only the fixed-point level of such Mackey functors.
It is, however, useful to note that for any $X\in\Sp^{h_{C_2}\t}$ we do have a restriction map of spectral sequences
\[\res\colon  \textrm{HHFPSS}_{C_2}(X)\to\textrm{HHFPSS}(i^\ast_e X),\] where we consider $i^\ast_e X$ as an ordinary spectrum with $\t$-action. This allows us to compare our computations to those of Bruner--Rognes. 

\subsection{A spectral sequence for geometric fixed points}\label{sec:gfpss}
In this section, we construct a variant of our spectral sequence geared towards applications to the geometric fixed points of Real trace invariants.

To this end, let $E\mu_2 = S(\mathbb{R}^\infty)$, and define 
\[E\mu_2^{(2n)} = E\mu_2^{(2n+1)} = S(\mathbb{R}^{n+1}).\]

\begin{construction}\label[construction]{cons:gfpSStower}
Let $Y$ be a spectrum with $\mu_2$-action. 
Applying $F(-,Y)^{\mu_2}$ to the filtration on $E\mu_2$ yields a tower of spectra \[\cdots\to F((E\mu_2)^{(2n)}_+,Y)^{\mu_2}\to F((E\mu_2)^{(2n-1)}_+,Y)^{\mu_2}\to \cdots F((E\mu_2)_+^{(0)},Y)^{\mu_2}\to\ast,\]
with homotopy limit $Y^{h\mu_2}$. Applying $\F_2$-homology yields a spectral sequence computing
\[H_\ast^c(Y^{h\mu_2}; \mathbb{F}_2) = \lim_n H_\ast(F( (E \mu_2^{(n)})_+, Y)^{\mu_2};\mathbb{F}_2).\] 
We set 
\[E^1_{s,t}\coloneq H_{\lfloor s/2\rfloor+t}\Big(F\big((E\mu_2)^{-s}_{-s}, Y\big)^{\mu_2};\mathbb{F}_2\Big)\cong \begin{cases} H_t (Y;\mathbb{F}_2) & \text{if }s=-2n \leq 0\\ 
0 & \text{else}.\end{cases}\]
Note that, as before, the spectral sequence is concentrated in even filtration degrees, implying $E^{2r-1} = E^{2r}$. Write $z^n\cdot x$ for the element in $E_{-2n,t}^2$ which corresponds to $x\in H_{t}(Y;\mathbb{F}_2)$. 
The resulting spectral sequence has signature
\[E^2_{-2s,t} = z^s\cdot H_t(Y;\mathbb{F}_2)\implies H_{t-s}^c(Y^{h \mu_2};\mathbb{F}_2),\]
and differentials
\[d^{2r}\colon E_{2s,t}^{2r}\to E^{2r}_{2s-2r,t+r-1}.\]  We refer to this as the \emph{homological homotopy fixed point spectral sequence} for $Y$, and 
denote it by $\mathrm{HHFPSS}(Y)$.
\end{construction}

\begin{remark}
 Note that we do not notationally distinguish between the homological homotopy fixed point spectral sequence for spectra with $\t$-action and spectra with $\mu_2$-action.
\end{remark}

\begin{example}
For $X\in\Sp^{h_{C_2}\t}$, we may view $X^{\Phi C_2}$ as a spectrum with $\mu_2$-action. 
If $X = \THR(A)$, then the $\mathrm{HHFPSS}(\THR(A)^{\Phi C_2})$ gives a way of computing the continuous homology, $H_\ast^c\left((\THR(A)^{\Phi C_2})^{h\mu_2};\mathbb{F}_2\right)$, from $H_\ast\left(\THR(A)^{\Phi C_2};\mathbb{F}_2\right)$.
\end{example}

\begin{proposition}\label[proposition]{prop:gfpssmap}
For any $X\in\Sp^{h_{C_2}\t}$, there is a map of spectral sequences 
\[\overline{\Phi}\colon \mathrm{HHFPSS}_{C_2}(X)\to \mathrm{HHFPSS}(X^{\Phi C_2}),\] which on $E^2$-pages sends $\bar{y}$ to $z$, and $x\in H_V(X;\mft )$ to the reduced geometric fixed points
$\overline{\Phi}(x)\in H_{|V^{C_2}|}(X^{\Phi C_2};\mathbb{F}_2)$. 
\end{proposition}
\begin{proof}
Consider the $a_\sigma$-localization of the tower of $C_2$-spectra from \cref{cons:SStower}: 
\[\cdots \to a_\sigma^{-1}F(E_{C_2}\t^{(2n)}_+, X)^\t \to a_\sigma^{-1}F(E_{C_2}\t^{(2n-1)}_+,X)^\t \to \cdots \to a_\sigma^{-1}F(E_{C_2}\t^{(0)}_+,X)^\t \to \ast.\]
Applying $(a_{\sigma}^{-1}H\mft \wedge - )^{C_2} \simeq (H\mft)^{\Phi C_2}\wedge (-)^{\Phi C_2}$, results in a tower of spectra
equipped with a canonical map from the tower of \cref{cons:SStower}. Identifying $(H\mft )^{\Phi C_2}\simeq H\mathbb{F}_2 [b]$, where $b =u_\sigma/a_\sigma$, we obtain a map $(H\mft)^{\Phi C_2}\to H \mathbb{F}_2$ which takes the quotient by the ideal generated by $b$. 
Taking the smash product with the exchange map 
\[ \left(F(E_{C_2}\t_+^{(\ast)},X)^\t\right)^{\Phi C_2}\to F\left((E\mu_2)_+^{(\ast)},X^{\Phi C_2} \right)^{\mu_2},\] gives a map from the localized tower to the tower giving rise to 
$\mathrm{HHFPSS}(X^{\Phi C_2})$.
\end{proof}

\begin{example}
Consider the case $X = S$, the $C_2$-equivariant sphere spectrum with trivial twisted $\t$-action. Then $S^{\Phi C_2}$ is the sphere with trivial $\mu_2$-action. The $\mathrm{HHFPSS}(S^{\Phi C_2})$ collapses at the $E^2$-page yielding an isomorphism $H_\ast^c \left((D\mathbb{R} P^\infty_+);\mathbb{F}_2 \right) \cong \F_2[z]$. The reduced geometric fixed point map 
$\overline{\Phi}\colon  H_\star[\bar{y}]\to\F_2[z]$ is the map which kills the negative cone and sends $\bar{y}\mapsto z, a_\sigma\mapsto 1, u_\sigma\mapsto 0$.  
\end{example}

\begin{remark}
In \cite{DMP24}, the authors give a formula for the geometric fixed points $\TCR(A;2)^{\Phi C_2}$ as an equalizer involving $\THR(A)^{\Phi C_2}$ and its genuine $\mu_2$-fixed points. 
Harpaz, Nikolaus, and Shah provide such a formula instead involving the homotopy fixed points $(\THR(A)^{\Phi C_2})^{h\mu_2}$ in forthcoming work. Thus, it may be possible to obtain information on $\TCR(A)^{\Phi C_2}$ using our spectral sequences. We hope to come back to this in future work.
\end{remark}

\section{Differentials}\label{Sec:Diff}

\subsection{A geometric description of the differentials in $\mathrm{HHFPSS}_{C_2}$}
In this section, we study the differentials in the homological parametrized homotopy fixed point spectral sequence. This generalizes  \cite[Sec. 3]{BR05} to the Real setting; our proof strategy follows that of Bruner and Rognes, but since some care with degrees and levels of equivariance is required, we write out the details. 

Let $x \in H_V(X;\mft) \cong E^2_{0,V}$ be represented by a $C_2$-equivariant map $S^V \to H\mft\wedge X$ and suppose $x$ survives to a nontrivial class in $E^{2r}_{0,V}$ for some $ r \geq 1$. By adjunction, $x$ is represented by a $C_2$-twisted $\t$-equivariant map
\[x \colon  S(\c)_+ \wedge S^V \to H\mft \wedge X\]
where $\t$ acts freely on $S(\c)$ and trivially on $S^V$ and $H\mft$. The hypothesis
\[d^2(x), \ \ldots,\ d^{2r-2}(x) = 0\]
is equivalent to $x$ being in the image of the map $H_V(F(S(\c^r)_+,X)^\t ;\mft) \to H_V(F(S(\c)_+,X)^\t;\mft)$  induced by the inclusion $S(\c)_+ \hookrightarrow S(\c^r)_+$, which is equivalent to the existence of a $C_2$-twisted $\t$-equivariant extension 
$$x' \colon  S(\c^r)_+ \wedge S^V \to H\mft \wedge X$$
of $x$ along $S(\c)_+ \hookrightarrow S(\c^r)_+$. 
The next possible nontrivial differential
$$d^{2r}(x) \in E^{2r}_{-2r, V + r \rho -1}$$
is the obstruction to the existence of a $C_2$-twisted $\t$-equivariant extension 
$$x'' \colon  S(\c^{r+1})_+ \wedge S^V \to H\mft \wedge X$$
of $x'$ along $S(\c^r)_+ \hookrightarrow S(\c^{r+1})_+.$ 

The right adjoints of these maps fit into the diagram 
\[
\begin{tikzcd}
	& S(\c)_+ \arrow{d} \arrow{dr}{x} & \\
(\t \times \partial D^{r\rho})_+ \arrow{r}{\alpha_+} \arrow{d} & S(\c^r)_+ \arrow{r}{x'} \arrow{d} & F(S^V,H\mft \wedge X) \\
(\t \times D^{r\rho})_+ \arrow{r} & S(\c^{r+1})_+ \arrow[dashed,swap, ur,"x''"] & 	
\end{tikzcd}
\]
which identifies the obstruction to the existence of $x''$ with the obstruction to extending $x' \circ \alpha_+$ over $(\t \times D^{r \rho})_+$, still as a $C_2$-twisted $\t$-equivariant map. By adjunction, this identifies with the obstruction to extending the underlying $C_2$-equivariant map $x' \colon  \partial D^{r \rho}_+ = S(\c^r)_+ \to F(S^V,H\mft \wedge X)$ over $D^{r \rho}_+$. Using the $C_2$-equivariant stable splitting $\partial D^{r \rho}_+ \simeq S^{r \rho - 1} \vee D^{r \rho}_+$, 
this is the restriction of the $C_2$-equivariant map $x'$ to the stable summand $S^{r \rho -1}$. The left adjoint of this restriction is
\[\delta x \colon  S^{r\rho -1} \wedge S^V \to H\mft \wedge X,\]
which represents $d^{2r}(x)$. We have thus argued the following.

\begin{lemma}
Let $x \in E^{2r}_{0,V} \subset H_V(X;\mft)$ be represented by a $C_2$-twisted $\t$-equivariant map $x\colon  S(\c)_+ \wedge S^V \to H\mft \wedge X$ that extends to a $C_2$-twisted $\t$-equivariant map $x' \colon  S(\c^r)_+ \wedge S^V \to H\mft \wedge X$. Then $d^{2r}(x) = \delta x \cdot \bar{y}^r$, where $\delta x \in H_{V + r\rho - 1}(X;\mft)$ is represented by $x'$ considered as a $C_2$-equivariant map, restricted to the stable summand $S^{r \rho-1} \wedge S^V$ of $S(\c^r)_+ \wedge S^V$. 
\end{lemma}

By adjunction, the map $x'$ represents a class in $H_\star \left( F(S(\c^r)_+,X)^\t ;\mft \right)$, and we may regard $x'$ as a $C_2$-equivariant map by composing with the forgetful map 
\begin{equation}\label{eqn:phi}
 \varphi\colon  F(S(\c^r)_+,X)^\t \to F(S(\c^r)_+,X).
\end{equation}
On the other hand, the canonical $C_2$-equivariant map
\begin{equation}\label{eqn:nu}
 \nu \colon  X \wedge D(S(\c^r)_+) \to F(S(\c^r)_+,X)
\end{equation}
is a $C_2$-equivariant weak equivalence because $S^{r \rho-1}_+$ may be equipped with the structure of a finite $C_2$-CW complex. Applying Bredon homology yields a natural isomorphism
\[\nu_\star \colon  H_\star(X ;\mft) \otimes H^{-\star}(S^{r \rho-1} ;\mft) \xrightarrow{\cong} H_\star(F(S^{r\rho-1}_+,X) ;\mft).\]
Let $\iota_{r\rho-1}$ denote the canonical generator of $H^{r\rho-1}(S^{r\rho-1};\mft)$.

\begin{proposition}
The composite map
\[H_\star(F(S(\c^r)_+,X)^\t;\mft) \xrightarrow{\varphi_\star} H_\star(F(S^{r\rho-1}_+,X);\mft) \xrightarrow{\nu_\star^{-1}} H_\star(X;\mft) \otimes H^{-\star}(S^{r\rho-1};\mft)\]
takes any class $x'$ that maps to $x \in E^{2r}_{0,V} \subset H_V(X)$ by the restriction map
\[H_\star(F(S(\c^r)_+,X)^\t;\mft) \to H_\star(F(S(\c)_+,X)^\t;\mft) = H_\star(X;\mft)\]
to the sum
\[(\nu^{-1}_\star \varphi_\star)(x') = x \otimes 1 + \delta x \otimes \iota_{r\rho-1},\]
where $d^{2r}(x) = \delta x \cdot \bar{y}^r $ in $E^{2r}_{-2r,V+r\rho-1}$. 
\end{proposition}
%

\subsection{Explicit formula for the $d^2$-differential}\label{SS:d2mu}
Choosing a $C_2$-equivariant identification $\t\cong S^\sigma$, 
the $C_2$-twisted $\t$-action on $X$ gives rise to an action map in $\Sp^{C_2}$ \[\alpha\colon S^\sigma_+\wedge X\to X.\] 
Let $\iota_\sigma\in H_\sigma(S^\sigma_+;\mft)$ be the fundamental class. 
Following Rognes in \cite[Sec. 3]{Rog98}, one can show that $d^2(x) = \bar{y}\cdot\bar{\sigma}(x)$, where $\bar{\sigma}(x) \equiv \alpha_\star(\iota_\sigma\otimes x)$ is the $C_{2}$-analogue of the Connes' operator.
In particular, for degree reasons, we have $\iota_\sigma^2 = 0$, and so 
\[\bar{\sigma}(\bar{\sigma}(x)) = \alpha_\star(\iota_\sigma\otimes\alpha_\star(\iota_\sigma\otimes x)) = \alpha_\star(\iota_\sigma^2\otimes x) = 0.\]
This mirrors the classical situation exactly, with the exception that $\bar{\sigma}$ is no longer a derivation as we now show. 

\begin{proposition}\label[proposition]{prop:twistedsigma}
Let $R\in\Sp^{h_{C_2}\t}$ be a ring.
The operator $\bar{\sigma}\colon H_\star(R;\mft)\to H_{\star+\sigma}(R;\mft)$ satisfies 
\[\bar{\sigma}(ab) = \bar{\sigma}(a)b + a\bar{\sigma}(b) + a_\sigma\bar{\sigma}(a)\bar{\sigma}(b).\]
\end{proposition}
\begin{proof}
Let $\mu\colon R\wedge R\to R$ denote the multiplication map and $\Delta\colon S^\sigma_+\to S^\sigma_+\wedge S^\sigma_+$ be the diagonal map.
Since the $C_2$-twisted $\t$-action on $R$ is by ring maps, the following diagram commutes:
\begin{center}
\begin{tikzcd}
S^\sigma_+\wedge R\wedge R\ar[d,"\Delta\wedge\id"]\ar[r,"\id\wedge\mu"] & S^\sigma_+\wedge R \ar[r,"\alpha"] & R\\
S^\sigma_+\wedge S^\sigma_+\wedge R\wedge R\ar[r] & S^\sigma_+\wedge R\wedge S^\sigma_+\wedge R\ar[ur,"\alpha\wedge\alpha" swap] &.
\end{tikzcd}
\end{center}
Finding a formula for $\bar{\sigma}(ab)$ then amounts to computing $\Delta(\iota_\sigma)\in H_\sigma(S^\sigma_+\wedge S^\sigma_+;\mft)$. 
A basis for this group is readily computed and we see that \[\Delta(\iota_\sigma) = k_1\iota_\sigma\otimes 1 + k_2 1\otimes\iota_\sigma + k_3 a_\sigma(\iota_\sigma\otimes\iota_\sigma),\] for some $k_i\in\F_2$. 
Since $\mathrm{pr}_i\circ\Delta = \id$, we have $k_1 = k_2 = 1$. Notice that $\Delta$ restricted to fixed points 
is the diagonal on a discrete set. Thus, applying geometric fixed points and naturality, one can verify that $k_3 = 1$. 
\end{proof}

Under the identification $\t\cong S^\sigma$, precomposing the action map with the Euler class 
\[S^0_+\wedge X\xrightarrow{a_\sigma\wedge \id_X} S^\sigma_+\wedge X\xrightarrow{\alpha} X\]
corresponds to restricting the action to the Weyl group $\mu_2\subset\t$. Thus, if $\iota_0,\iota_\infty\in H_0(S_+^0;\mft)$ are the generators corresponding to the identity and nonidentity elements in $\mu_2$, respectively, 
then $(a_\sigma)_\star(\iota_0+\iota_\infty)$ is a nonzero element in the kernel of restriction, and so it must be 
$a_\sigma\iota_\sigma$. Then we can define \[\bar{\sigma}_\mu(x) \coloneq \alpha_\star((a_\sigma)_\star(\iota_0+\iota_\infty)\otimes x) = a_\sigma\bar{\sigma}(x).\]
As a consequence of \cref{prop:twistedsigma}, we have \[\bar{\sigma}_\mu(ab) = \bar{\sigma}_\mu(a)b + a\bar{\sigma}_\mu(b) + \bar{\sigma}_\mu(a)\bar{\sigma}_\mu(b).\]
The $d^2$-differential then takes the following form. 
\begin{proposition}\label[proposition]{prop:d2formula}
For any $x\in E^2_{0,V} = H_V(X;\mft)$, we have 
\[d^2 x = \bar{y}\cdot\bar{\sigma}(x) = \bar{y}\cdot\frac{\bar{\sigma}_\mu(x)}{a_\sigma}.\]
\end{proposition}

Consider the map $\mu\colon X\to X$ obtained by restricting the action map $\alpha$ to the nontrivial element in the Weyl group. 
From a homology calculation, one can verify that the image of $x\in H_\star(X;\mft)$ under the map induced by $\mu$ is exactly \[\mu_\star(x) = x+\bar{\sigma}_\mu(x).\] 
The formula of \cref{prop:twistedsigma} can then be written as \[\bar{\sigma}(ab) = \bar{\sigma}(a)b + \mu_\star(a)\bar{\sigma}(b).\] So one can think of $\bar{\sigma}$ as a 
Weyl-twisted derivation. 

\begin{remark}
We remark that an alternate proof of \cref{prop:twistedsigma} may be obtained from the fact that $\mu_\star$ acts by ring maps in that context. 
\end{remark}

\subsection{Differentials in the $\mu_2$-spectral sequence.}
We now briefly consider the differentials in the spectral sequence of \cref{sec:gfpss}.

For $Y$ a spectrum with $\mu_2$-action, we write $\mu_\ast \colon H_\ast(Y;\mathbb{F}_2)\to H_\ast(Y;\mathbb{F}_2)$ for the induced action, and 
$\sigma_\mu(x) = x +\mu_\ast(x)$. 

\begin{proposition}\label[proposition]{prop:gfpE4pg}
The $E^4$-page of the $\mathrm{HHFPSS}(Y)$ is given by 
\[E^4_{\ast,\ast}\cong H^\ast(B\mu_2; H_\ast (Y;\mathbb{F}_2)).\] 
\end{proposition}
\begin{proof}
Following the geometric arguments above, one finds a formula for the $d^2$-differential as 
before $d^2(x) = z\cdot \sigma_\mu(x)$.  
\end{proof}

Note that for $X\in\Sp^{h_{C_2}\t}$, the geometric fixed points of the twisted action map, $\alpha$, gives a $\mu_2$-action map 
on $X^{\Phi C_2}$. It follows that these actions are compatible in the sense that \[\overline{\Phi}(\bar{\sigma}x) = \sigma_\mu\overline{\Phi}(x).\]

\section{Multiplicative structure}\label{Sec:Mult}
\subsection{Multiplication and Leibniz rule}

We now extend \cite[Sec. 4]{BR05} to the Real setting. Before we begin, we discuss the following subtlety:

\begin{warning}
There are different levels of equivariant multiplicative structure which can be present in spectral sequences in equivariant stable homotopy theory. One form directly mirrors classical multiplicative structure: if a spectral sequence arises from a multiplicative (in the classical sense) filtration, then it will admit a Leibniz rule and we can interpret expressions like ``the spectral sequence for $M$ is a module over the spectral sequence for $R$" when $M$ is an $R$-module. In the $C_2$-equivariant setting, there is a richer form of multiplicative structure, which might be called genuine equivariant multiplicative structure, which refines the ordinary multiplicative structure. In a genuine equivariant multiplicative structure, one not only has a Leibniz rule for ordinary products, but also a ``Leibniz rule" for norms, i.e., $C_2$-indexed products. This sort of structure is present in the equivariant slice spectral sequence and produces the ``gold relation" (cf. \cite{HHR17}). 

We will show in this section that our spectral sequence admits a naive multiplicative structure, but we do not expect it to admit a genuine equivariant multiplicative structure. It turns out that this weaker level of multiplicative structure, plus some compatibility with equivariant Dyer--Lashof operations, suffices for our purposes. 
\end{warning}

We will now make precise the multiplicative structure present in the homological parametrized homotopy fixed point spectral sequence.
Let \(X=\colim_n X^{(n)}\) with \(X^{(2n)}= X^{(2n+1)}=S(\mathbb{C}^{n+1})\), so that we have a cofiber sequence:
\[X^{(2n-1)} \to X^{(2n)} \to \mathbb{T}_+ \wedge S^{n\rho}.\]
Let \(Y=X\times X\) be filtered by
\[
Y^{(s)} \;=\; \bigcup_{a+b\le s} X^{(a)} \times X^{(b)}\qquad(s\ge 0).
\]
On the product $Y$, the \(n\)-th filtration layer is obtained from \(Y^{(n-1)}\) by attaching product $C_2$-slice cells
\[
( \mathbb{T} \times D^{a\rho} ) \times (\mathbb{T} \times D^{b\rho}) \cong  \mathbb{T} \times \mathbb{T} \times D^{(a+b)\rho}.
\]
where $a+b = n$

\begin{lemma}\label[lemma]{lemma: filtered diagonal up to homotopy}
The diagonal map $\Delta\colon X\to Y$ is $O(2)$-equivariantly homotopic to a filtered map $d$ satisfying
\(d(X^{(n)})\subset Y^{(n)}\) for all \(n\ge 0\).
\end{lemma}

\begin{proof}
We argue by induction on \(n\). The base case is immediate since $Y^{(0)} = X^{(0)} \times X^{(0)}$, so $\Delta|_{X^{(0)}}$ already lands in $Y^{(0)}$.
Now assume inductively that, for some $2n-1\ge0$, we have constructed a filtered map $d_{2n-1} \colon  X^{(2n-1)} \to Y^{(2n-1)}$ together with an \(O(2)\)-equivariant homotopy
\(h_{2n-1}\colon X^{(2n-1)}\times I\to Y\) from \(\Delta|_{X^{(2n-1)}}\) to \(\iota_{ 2n-1 }\circ d_{ 2n-1 }\), where \(\iota_{2n-1}\colon Y^{(2n-1)}\hookrightarrow Y\) is the inclusion.
Let \(\alpha\colon  \mathbb{T} \times S^{n\rho-1} \to X^{(2n-1)}\) be the attaching map and \(\tilde \alpha\colon \mathbb{T} \times D^{n\rho}\to X\) the characteristic map to obtain $X^{(2n)}$.
We need to homotope $\Delta \tilde \alpha$ so that it lands in $Y^{(2n)}$, compatibly with the homotopy already chosen on the boundary.
This is a local extension problem over the new cell. 
Thus, for the following diagram, we replace \((X^{(2n)},X^{(2n-1)})\) by the pair $(\mathbb T\times D^{n\rho},\mathbb T\times S^{n\rho-1})$ through the characteristic map \(\tilde\alpha\) and its restriction \(\alpha\).
To find the homotopy $h_{2n}$, consider the following diagram:
\[\xymatrix{
\mathbb{T} \times S^{n\rho -1}  \ar[rr] \ar[dd] && \mathbb{T} \times S^{n\rho -1}  \times I \ar[dd]|-\hole \ar[dl]_-{h_{2n-1} (\alpha \times\id)} && \mathbb{T} \times S^{n\rho -1} \ar[ll] \ar[dd]  \ar[dl]_-{d_{2n-1} \alpha}\\
 & Y & & Y^{(2n)} \ar[ll]_(.4){\iota_{2n}} \\
 \t \times D^{n\rho}  \ar[rr]  \ar[ur]^{\Delta \tilde{\alpha}} &&  \mathbb{T} \times D^{n\rho}  \times I \ar@{..>}[ul]_-{h_{2n}} &&  \mathbb{T} \times D^{n\rho}  \ar[ll] \ar@{..>}[ul]_-{d_{2n}}
}\]
We claim the dashed lifting exists.

These dashed maps exist if the following set of $O(2)$-equivariant homotopy classes is trivial:
\[ [\mathbb{T} \times S^{n\rho-1}, \mathrm{Fib} (\iota_{2n})]^{O(2)} = \ast.\]
We may write these homotopy classes as relative homotopy classes:
\[[\mathbb{T} \times S^{n\rho-1}, \mathrm{Fib} (\iota_{2n})]^{O(2)} = [ (\mathbb{T} \times D^{n\rho} , \mathbb{T} \times S^{n\rho-1}) , (Y, Y^{(2n)}) ]^{O(2)}_.\]
Consider a map $f\colon  (\mathbb{T} \times D^{n\rho} , \mathbb{T} \times S^{n\rho-1}) \to (Y, Y^{(2n)})$.
Then the image of $f$ lies in some finite filtration level, say $Y^{(k)}$ for some $k \ge 2n$.
If $k=2n$, we are done, so suppose $k > 2n$. By the equivariant smooth (or simplicial) approximation theorem, there is a map $f'\colon \mathbb{T} \times D^{n\rho} \to Y^{(k)}$ such that $f' = f$ on $ \mathbb{T} \times S^{n\rho-1}$, $f' \simeq f$ (rel $\mathbb{T} \times S^{n\rho-1}$) and $f'$ misses a fixed point $p$ in the interior of the product of $C_2$-slice cells 
\[( \mathbb{T} \times D^{a\rho}) \times ( \mathbb{T} \times D^{b\rho}) \cong \mathbb{T} \times\mathbb{T} \times D^{k\rho},\]
where $a+b = k$.
Therefore, we can deform $f'$ to a map into $Y^{(k-1)}$.
By iterating this procedure, we can homotope $f$ (rel $\mathbb{T} \times S^{n\rho-1}$) so that its image lies in $Y^{(2n)}$.
\end{proof}

\begin{remark}
Our work in Lemma~\ref{lemma: filtered diagonal up to homotopy} shows there is a filtered diagonal map exhibiting $E_{C_2}\TT^{(2\ast)}$ as a homotopy associative coalgebra in the category of $O(2)$-equivariant filtered spaces. 
Applying $(-)_{h_{C_2}\TT}$ to this filtration, we find that $\CP^\ast$ is a homotopy associative coalgebra in the category of filtered $C_2$-spaces, where the $C_2$-action is given by complex conjugation. 
In fact, one can show this coalgebra structure refines to an $\mathbb{E}_{\rho}$-coalgebra, $\CP_{\mathbb{R}}^{\ast}$, in the \category of filtered $C_{2}$-spaces. 
The argument is almost identical to Lurie's from \cite[5.2]{lurieRotationInvarianceAlgebraic2015} but at crucial moments one uses results of Horev \cite{horev2019genuineequivariantfactorizationhomology} and Horev--Klang--Zou \cite{Horev-Klang-Zou}.
Furthermore, it is interesting to note that $\CP^{\ast}_{\mathbb{R}}$ does not seem to admit any further equivariant comultiplicative structure. 


\end{remark}

\begin{construction}
Fix integers $s,s'$ and $r\ge 1$, and set $m\coloneq-s-s'+r-1$, and consider the composition
\begin{align*}
E_{C_2}\t^{(m)} &\xrightarrow{\,d\,} (E_{C_2}\t\times E_{C_2}\t)^{(m)} 
\hookrightarrow E_{C_2}\t^{(m)}\times E_{C_2}\t^{(m)},
\end{align*}
where $E_{C_2}\t\times E_{C_2}\t$ is given the product filtration.
Set
\[
A\coloneq E_{C_2}\t^{(-s+r-1)},\quad
B\coloneq E_{C_2}\t^{(-s'+r-1)},\quad
A_0\coloneq E_{C_2}\t^{(-s-1)},\quad
B_0\coloneq E_{C_2}\t^{(-s'-1)},
\]
and let
\[
U\coloneq(A_0\times E_{C_2}\t^{(m)})\cup(E_{C_2}\t^{(m)}\times B_0).
\]
\begin{figure}[ht]
\centering
\def\aval{20} 
\def\bval{16} 
\def\rval{12}
\pgfmathsetmacro{\mval}{\aval+\bval+\rval-1}

\pgfmathsetlengthmacro{\figsize}{6.2cm}
\pgfmathsetlengthmacro{\unit}{\figsize/(\mval+1.8)}

\begin{tikzpicture}[x=\unit,y=\unit,>=stealth, every node/.style={font=\footnotesize},scale = 0.7]

\pgfmathsetmacro{\xmax}{\mval+1.6}
\pgfmathsetmacro{\ymax}{\mval+1.6}

\draw[->] (0,0) -- (\xmax,0);
\draw[->] (0,0) -- (0,\ymax);

\draw[thick] (0,0) rectangle (\mval,\mval);

\fill[red, fill opacity=0.15]
  (0,0) --
  (\mval,0) --
  (\mval,\bval-1) --
  (\aval-1,\bval-1) --
  (\aval-1,\mval) --
  (0,\mval) -- cycle;

\filldraw[fill=green, fill opacity=0.12, draw=green!50!black]
  (0,0) rectangle (\aval+\rval-1,\bval+\rval-1);

\node[red!70!black] at ({.8*(\aval-1)}, {0.65*(\bval-1+\mval)}) {$U$};

\node[green!60!black] at ({\aval+\rval-1}, {1.05*(\bval+\rval-1)}) {$A\times B$};

\draw[decorate,decoration={brace,mirror,amplitude=4pt}]
  (0,-0.8) -- (\aval-1,-0.8)
  node[midway,below=6pt] {$A_0$};

\draw[decorate,decoration={brace,mirror,amplitude=4pt}]
  (0,-1.7) -- (\aval+\rval-1,-1.7)
  node[midway,below=6pt] {$A$};

\draw[decorate,decoration={brace,amplitude=4pt}]
  (-0.8,0) -- (-0.8,\bval-1)
  node[midway,left=6pt] {$B_0$};

\draw[decorate,decoration={brace,amplitude=4pt}]
  (-1.7,0) -- (-1.7,\bval+\rval-1)
  node[midway,left=6pt] {$B$};

\draw[thick,blue] (\mval,0) -- (0,\mval);
\draw[thick,blue] (\mval-\rval,0) -- (0,\mval-\rval);

\draw (\mval,0) -- ++(0,-3pt) node[below=3pt] {$m$};
\draw (0,\mval) -- ++(-3pt,0) node[left=4pt] {$m$};

\draw (0,\mval-\rval) -- ++(-3pt,0) node[left=4pt] {$-s-s'-1$};

\end{tikzpicture}
\caption{Schematic  of the filtration region.}
\label{fig:filtered-diagonal}
\end{figure}
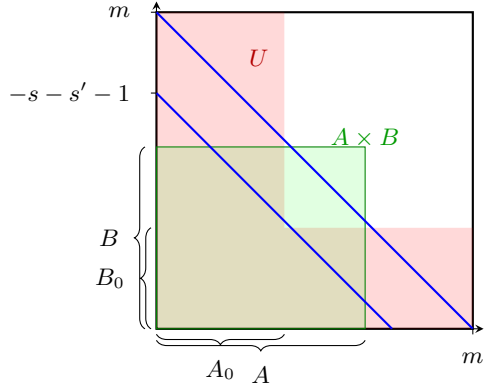

\cref{fig:filtered-diagonal} shows that \(d(E_{C_2}\t^{(-s-s'-1)})\subset U\) and \(d(E_{C_2}\t^{(m)})\subset (A\times B)\cup U\). Hence \(d\) induces a map of pairs
\[
d\colon \bigl(E_{C_2}\t^{(m)},E_{C_2}\t^{(-s-s'-1)}\bigr)
\longrightarrow
\bigl((A\times B)\cup U,U\bigr).
\]
Passing to quotients gives
\[
\frac{E_{C_2}\t^{(m)}}{E_{C_2}\t^{(-s-s'-1)}} \longrightarrow \frac{ (A\times B)\cup U}{U} \cong \frac{A \times B}{(A \times B)\cap U}.
\]
Since \((A\times B)\cap U=(A_0\times B)\cup(A\times B_0)\), the map of pairs above induces a map
\begin{equation}\label{map_d}
E_{C_2}\t^{-s-s'+r-1}_{-s-s'}
\longrightarrow
E_{C_2}\t^{-s+r-1}_{-s}\wedge E_{C_2}\t^{-s'+r-1}_{-s'}.
\end{equation}
We continue to denote this map by \(d\).
\end{construction}

When $s = -2n$ and $s'= -2n'$ are non-positive and even, and $r=1$, this specializes to
\begin{equation}\label{eqn:diagonal}
d_{n,n'}\colon \t_+ \wedge S^{(n+n')\rho} \to \t_+ \wedge S^{n\rho} \wedge \t_+ \wedge S^{n' \rho},
\end{equation}
which is simply the projection of the map 
\[E_{C_2}\t^{2(n+n')}_{2(n+n')-1}\to (E_{C_2}\t\times E_{C_2}\t)^{2(n+n')}_{2(n+n')-1}\]
 onto a single product $\t$-cell.
In addition to the ordinary diagonal $\Delta\colon \t\to\t\times\t$, we consider the \emph{Weyl-twisted diagonal}
\[\Delta^\mu\colon \t\to\t\times\t,\quad z\mapsto (-z,z).\]
As maps of $\t$-equivariant spaces, $d_{n,n'}$ is homotopic to a suspension of $\Delta$, while $O(2)$-equivariantly, we instead have the following. 

\begin{lemma}\label{lemma:WeylTwistedDiagonal}
We have stable equivalences 
\[ d_{1,0} \simeq\Delta\wedge\id_{S^\rho},\]
up to the appropriate swap map, and
\[ d_{0,1} \simeq\Delta^\mu\wedge\id_{S^\rho}.\]
\end{lemma} 
\begin{proof}
We emphasize that we work stably throughout this proof and so all spaces appearing denote the corresponding suspension spectra. 

By adjunction, the $O(2)$-equivariant maps $d_{1,0}$ and $d_{0,1}$ are determined up to homotopy by the $C_2$-equivariant maps $S^\rho\to \t_+\wedge\t_+\wedge S^\rho$. From the $C_2$-equivalence $\t\simeq S^\sigma$, we see that such maps are stably classified up to homotopy by 
$\pi_0^{C_2}(S^\sigma\times S^\sigma)_+$. We claim that an element $x$ in this stable homotopy group is determined by 
its restriction $\res(x)\in\pi_0 (S^1\times S^1)_+$ and geometric fixed points $\Phi(x)\in\pi_0(S^0\times S^0)_+$. 

To see this, consider the tom Dieck splitting 
\[(S^\sigma\times S^\sigma)_+^{C_2}\simeq (S^0\times S^0)_+\vee (S^\sigma_+\wedge S^\sigma_+)_{hC_2}.\]
Since $S^\sigma_+\simeq S^0\vee S^\sigma$, the second wedge summand above can be identified with a wedge of stunted projective spaces 
\[(S^\sigma_+\wedge S^\sigma_+)_{hC_2}\simeq \mathbb{R} P^\infty_+\vee \mathbb{R} P_1^\infty\vee\mathbb{R} P_1^\infty \vee\mathbb{R} P_2^\infty.\] 
So we have $\pi_0^{C_2}(S^\sigma\times S^\sigma)_+\cong\zz^5$, where four of the copies of $\zz$ correspond to $\pi_0$ of the geometric fixed points. It remains to analyze the 
last copy of $\zz$. For this, we consider the crush map $p:(S^\sigma\times S^\sigma)_+\to S^0$. Naturality of the tom Dieck splitting implies that the last copy of $\zz$ above maps 
isomorphically onto the second copy of $\zz$ in 
\[\pi_0^{C_2} S^0\cong \pi_0 (S^0 \vee \mathbb{R} P^\infty_+)\cong\zz\oplus\zz.\] Elements $(a,b)\in\pi_0^{C_2}S^0$
are determined by their underlying and geometric fixed points since $\res(a,b) = a+2b$ and $\Phi(a,b)=a$. The claim for $x\in\pi_0^{C_2}(S^\sigma\times S^\sigma)_+$ then follows.

Thus it remains to check that the equivalences in the lemma hold upon restriction to underlying and geometric fixed points. 
The claim on underlying is in the proof of \cite[Proposition~4.1]{BR05}. To check this on geometric fixed points, we note that the $O(2)$-equivariant filtered diagonal $d$ from \cref{lemma: filtered diagonal up to homotopy} induces a $\mu_2$-equivariant filtered diagonal for $S(\mathbb{R}^\infty)$ upon 
taking $C_2$-fixed points. Without loss of generality, we assume that d takes the form shown below in \cref{fig:diagonalpic}. The result is then readily verified. 
\end{proof}

\begin{figure}[H]
\begin{tikzpicture}[
    scale=4,
    >={Stealth[length=2.2mm]},
    line cap=round,
    line join=round
]



\draw[black, line width=0.7pt]
    (0,0) -- (1,0) -- (1,1) -- (0,1) -- cycle;

\draw[black, line width=0.7pt]
    (0,0.5) -- (1,0.5);

\draw[black, line width=0.8pt]
    (0,0) -- (0,1);
\draw[black, line width=0.8pt]
    (0.5,0) -- (0.5,1);
\draw[black, line width=0.8pt]
    (1,0) -- (1,1);

\draw[blue!70!black, line width=0.8pt]
    (0,0) -- (1,1);

\draw[red!70!black, line width=0.8pt]
    (0,0) -- (0.5,0) -- (0.5,0.5) -- (1,0.5) -- (1,1);

\draw[blue!70!black, ->, line width=0.8pt]
    (0.18,0.18) -- (0.32,0.32);
\draw[blue!70!black, ->, line width=0.8pt]
    (0.68,0.68) -- (0.82,0.82);

\draw[red!70!black, ->, line width=0.8pt]
    (0.18,0) -- (0.32,0);
\draw[red!70!black, ->, line width=0.8pt]
    (0.5,0.18) -- (0.5,0.33);
\draw[red!70!black, ->, line width=0.8pt]
    (0.68,0.5) -- (0.83,0.5);
\draw[red!70!black, ->, line width=0.8pt]
    (1,0.68) -- (1,0.83);

\foreach \x/\y in {
    0/0, 0/0.5, 0/1,
    0.5/0, 0.5/0.5, 0.5/1,
    1/0, 1/0.5, 1/1
}{
    \fill[black] (\x,\y) circle (0.015);
}

\node[below] at (0,0) {$1$};
\node[below] at (0.5,0) {$-1$};
\node[below] at (1,0) {$1$};

\node[left] at (0,0) {$1$};
\node[left] at (0,0.5) {$-1$};
\node[left] at (0,1) {$1$};

\end{tikzpicture}
\caption{A deformation (red) of the diagonal (blue) for $S(\mathbb{R}^2)$.}
\label{fig:diagonalpic}
\end{figure}

\begin{proposition}\label[proposition]{Prop:MultStr}
Let $R\in\Sp^{h_{C_2}\t}$ be an $\mathbb{E}_\infty$-algebra. The homological parametrized homotopy fixed point spectral sequence for $R$ is a multiplicative spectral sequence of $\mathcal{A}_\star^{C_2}$-comodule algebras with 
\[E^2_{\ast,\star} = H_\star(R;\mft)\langle \bar{y}\rangle\implies H_\star^c(R^{h_{C_2}\t} ; \mft),\]
 where $H_\star(R;\mft)$ has the Pontryagin product, and the angle brackets indicate that there is a multiplicative relation \[\bar{y}a = \mu_\star(a)\bar{y}.\] 
If the spectral sequence converges strongly, then the isomorphisms $E^\infty_{s,\star}\cong F_{s,\star}/F_{s-1,\star}$ are isomorphisms of 
$\mathcal{A}_\star^{C_2}$-comodule algebras. 
\end{proposition}

\begin{proof}
The overall strategy of the proof follows that of Bruner--Rognes \cite{BR05}, with the main additional inputs being the two lemmas above.
For completeness, we recall the argument, emphasizing the points where the additional structure enters in this genuine $C_2$-equivariant context.

We first establish the multiplicative relation on the $E^2$-page. 
Letting $m$ denote the multiplication on $H\mft \wedge R$, the multiplication on the $E^2$-page is defined by the composite
\begin{equation*}
\begin{split} 
E^2_{-2i,V}\otimes E^2_{-2j,V'}\cong & \pi_{V-i\rho}F(\t_+\wedge S^{i\rho}, H\mft \wedge R)^\t\otimes \pi_{V'-j\rho}F(\t_+\wedge S^{j\rho},H\mft \wedge R)^\t\\
&\qquad\xrightarrow{-\wedge-}\pi_{V+V'-n\rho}F(\t_+\wedge S^{i\rho}\wedge\t_+\wedge S^{j\rho},H\mft \wedge R\wedge H\mft \wedge R)^\t\\
&\qquad\qquad\qquad\xrightarrow{d_{i,j}^\ast m_\ast}\pi_{V+V'-n\rho}F(\t_+\wedge S^{n\rho},H\mft \wedge R)^\t  \cong E_{-2n,V+V'}^2,
\end{split}
\end{equation*}
for $i+j=n$. Since $\t$ acts by conjugation on function spaces, 
\[(\id_{H\mft }\wedge \gamma)_\ast\simeq (-1\cdot)^\ast\colon F(\t_+,H\mft \wedge R)^\t\to F(\t_+,H\mft \wedge R)^\t,\] where 
$\gamma\colon R\to R$ denotes the action by the nontrivial element of $\mu_2\subset\t$. 
Thus, since $\Delta\simeq ((-1\cdot) \wedge\id_{\t})\circ \Delta^\mu$, we have a commutative diagram 
\begin{center}
\begin{tikzcd}[column sep = 5em]
E^2_{0,V}\otimes E^2_{-2,V'}\ar[r,"(\id_{H\mft }\wedge\gamma)_\ast\otimes 1"]\ar[d,"\text{swap}", swap] & E^2_{0,V}\otimes E^2_{-2,V'}\ar[d] \\
E_{-2,V'}^2\otimes E_{0,V}^2\ar[r]  & E_{-2,V+V'}^2,
\end{tikzcd}
\end{center}
where the unlabeled arrows are the multiplication maps defined as above. This proves the multiplicative relation on $E^2$. 

To organize the rest of the multiplicative structure, we define a multiplicative Cartan--Eilenberg system of $RO(C_2)$-graded abelian groups by
\[H(a,b) \coloneq H_\star(F(E_{C_2}\t^{b-1}_a,R)^\t;\mft)\]
for integers $a\le b$.
The structure maps $\delta$ and $\eta$ of this system are induced by the filtration of $E_{C_2}\t$, and we note that $E^1_s = H(-s,-s+1)$.

Applying $F(-,R)^\t$ and $\mft$-homology to the map~\eqref{map_d}, we obtain pairings
\[
\varphi_r \colon
H(-s,-s+r) \otimes H(-s',-s'+r)
\longrightarrow
H(-s-s',-s-s'+r)
\]
for all integers $s,s'$ and $r\ge 1$.

For $s=-2n$, $s'=-2n'$ and $r=1$, this yields a homomorphism
\[
\varphi_1 \colon E^1_{-2n} \otimes E^1_{-2n'}
\longrightarrow
E^1_{-2n-2n'}
\]
which sends
\[
\bar{y}^n \cdot x \otimes \bar{y}^{n'} \cdot x'
\longmapsto
\bar{y}^{n+n'} \cdot \mu_\ast(x \otimes x'),
\]
where $\mu_\ast$ denotes the Pontryagin product on $H_\star(R;\mft)$, under the identification $E^1_{-2n,\star} = \bar{y}^n \cdot H_\star (R;\mft)$ from (\ref{ssE1}).
Under this identification, it follows that $E^1 = E^2$, equipped with $\varphi_1$, is isomorphic to the tensor product $H_\star[\bar{y}] \otimes_{H_\star} H_\star(R;\mft)$.

To show that each $E^{r+1}$ inherits an algebra structure from $E^r$, we show that the differentials $d^r$ are derivations. For $r = 2$, this follows from \cref{prop:twistedsigma} and Weyl-twisted multiplication on the $E^2$-page proven above. 
To show this for $r>2$, it suffices to verify that this Cartan--Eilenberg system is multiplicative.
Explicitly, the pairings $\varphi_r$ must satisfy
\[\delta(\varphi_r(z\otimes z')) = \varphi_1 (\delta(z) \otimes \eta(z')) + (-1)^{|z|} \varphi_1 (\eta(z) \otimes \delta(z'))\]
in $E^1_{s+s'-r}$, for $z\in H(-s,-s+r)$ and $z' \in H(-s', -s'+r)$, where $|z|$ is the trivial representation  degree of $z$.
Here
\[ \delta \colon H(-s,-s+r) \longrightarrow H(-s+r,-s+r+1)=E^1_{s-r} \]
denotes the degree $(-1)$ homomorphism induced by the stable connecting map
\[ \delta \colon \Sigma^{-1} E_{C_2}\t^{-s+r}_{-s+r} \longrightarrow E_{C_2}\t^{-s+r-1}_{-s} \]
associated to the triple
\[ \left(E_{C_2}\t^{(-s+r)},\,E_{C_2}\t^{(-s+r-1)},\,E_{C_2}\t^{(-s-1)}\right). \]
Similarly,
\[ \eta \colon H(-s,-s+r) \longrightarrow H(-s,-s+1) \]
is the homomorphism induced by the inclusion
\[ \eta \colon E_{C_2}\t^{-s}_{-s} \longrightarrow E_{C_2}\t^{-s+r-1}_{-s}. \]

Phrased more geometrically, the multiplicativity condition amounts to showing that, in the diagram below, the composite along the diagonal is $O(2)$-equivariantly homotopic to the sum of the two outer composites:
\[\xymatrix@C=0em@R=1em{
\Sigma^{-1} E_{C_2}\t^{-s-s'+r}_{-s-s'+r} \ar[dr]^{\delta} \ar[rr]^-d \ar[dd]_-{(-1)^{\lfloor s/2\rfloor} d}& & \Sigma^{-1} E_{C_2}\t^{-s+r}_{-s+r} \wedge E_{C_2}\t^{-s'}_{-s'} \ar[dd]^{\delta\wedge\eta} \\
&E_{C_2}\t^{-s-s'+r-1}_{-s-s'}  \ar[dr]^{d}  & \\
E_{C_2}\t^{-s}_{-s} \wedge \Sigma^{-1} E_{C_2}\t^{-s'+r}_{-s'+r} \ar[rr]^-{\eta \wedge \delta}& & E_{C_2}\t^{-s+r-1}_{-s} \wedge E_{C_2}\t^{-s'+r-1}_{-s'.} 
}\]

As in the classical argument, it suffices to consider the case in which $s$, $s'$, and $r$ are all even, with $s$ and $s'$ non-positive.
In this case, the common source of the $O(2)$-equivariant stable maps under comparison is then $\Sigma^{-1} O(2)_+ \wedge_{C_2} S^{\tfrac{-s-s'+r}{2} \rho}$.
By adjunction,  the problem reduces to comparing $C_2$-equivariant maps
from the representation sphere $S^{\tfrac{-s-s'+r}{2} \rho-1}$ to the target
\[E_{C_2}\t^{-s+r-1}_{-s} \wedge E_{C_2}\t^{-s'+r-1}_{-s'} \simeq \left(S^{ \tfrac{-s+r}{2}\rho-1} \vee S^{\tfrac{-s}{2}\rho}\right) \wedge \left(S^{ \tfrac{-s'+r}{2}\rho-1} \vee S^{\tfrac{-s'}{2}\rho}\right).\]
The projection to the summand $S^{\tfrac{-s}{2}\rho} \wedge S^{\tfrac{-s'}{2}\rho}$ in the target is trivial for each of the three maps, since in each case the composite factors through two subsequent maps in a cofiber sequence.
The projection to $S^{(\frac{-s-s'}{2}+r)\rho-2}$ is necessarily trivial for $r>2$, which is sufficient, as noted above. 
The projection to the remaining summands
\[E_{C_2}\t^{(-s+r -1)} \wedge \Sigma E_{C_2}\t^{(-s' -1)} = S^{\tfrac{-s+r}{2}\rho -1} \wedge S^{\tfrac{-s'}{2}\rho}\]
and
\[\Sigma E_{C_2}\t^{(-s -1)} \wedge  E_{C_2}\t^{(-s'+r -1)} = S^{\tfrac{-s}{2}\rho } \wedge S^{\tfrac{-s'+r}{2}\rho-1}\]
agree as required, by the same type of homotopy as in \cref{lemma:WeylTwistedDiagonal}.

Lastly, the argument given at the bottom of \cite[pp. 670]{BR05} shows that when Boardman's obstruction group $RE^{\infty}_{\ast,\star}$ vanishes, the isomorphisms $F_s/F_{s-1} \cong E^{\infty}_s$ identify the associated graded algebra induced by the product on $H^c_\star(R^{h_{C_2}\t} ;\mft)$ with the algebra structure on the $E^{\infty}$-term. 
%
%
%
\end{proof}

The homological parametrized homotopy fixed point spectral sequence for any $C_2$-spectrum with twisted $\t$-action will be a module over the spectral sequence for the $C_2$-equivariant sphere spectrum with trivial $\t$-action. The latter spectral sequence collapses at $E^2$, so we have the following equivariant analogue of \cite[Lem. 4.3]{BR05}.

\begin{lemma}\label[lemma]{lemma:ystructure}
Let $X$ be any $C_2$-spectrum with twisted $\t$-action. 
\begin{enumerate}

\item The differentials in the homological parametrized homotopy fixed point spectral sequence converging to $H^c_\star(X^{h_{C_2}\t};\mft)$ satisfy the relation
$$d^{2r}(\bar{y}^n \cdot x) = \bar{y}^n \cdot d^{2r}(x)$$
for all $x \in E^{2r}_{0,\star} \subset H_\star(X;\mft)$, $2r \geq 2$, and $n \geq 0$. 

\item Each class in $E^{2r}_{-2n,V}$ has the form $\bar{y}^n \cdot x$ for a class $x \in E^{2r}_{0,V} \subset H_V(X;\mft)$. Hence the spectral sequence is completely determined by the differentials that originate in filtration zero. 

\item The $\bar{y}$-torsion in $E^{2r}_{\ast,\star}$ has height strictly less than $r$, and is concentrated in filtrations $-2r+4 \leq s \leq 0$. 

\end{enumerate}
\end{lemma}

\subsection{Interaction with equivariant Dyer--Lashof operations}

Wilson \cite[Def. 4.3]{Wil17} defined equivariant Dyer--Lashof operations 
\[Q^{k\rho} \colon  H_{V} (X;\mft) \to H_{V+k\rho} (X;\mft), \quad Q^{k\rho-1} \colon  H_{V}(X;\mft) \to H_{V+k\rho-1} (X;\mft) \]
with various convenient properties \cite[Thm. 4.4]{Wil17}. These equivariant Dyer--Lashof operations are natural with respect to maps of $C_2$-$\mathbb{E}_\infty$-algebras. We will show that they interact nicely with our spectral sequence provided that $R$ is a $C_2$-$\mathbb{E}_\infty$-algebra. To begin, note that by Proposition~\ref{proposition: mapping spectra E_oo} the filtered $C_2$-twisted $\TT$-spectrum
\[
F(E_{C_2}\TT^{(\ast)}_{+},R) 
\]
refines to a filtered object in $\CAlg^{C_2}(\Sp^{h_{C_2}\TT})$. 

\begin{proposition}\label[proposition]{prop:diffsoperations}
Let $R$ be a $C_2$-$\mathbb{E}_\infty$-algebra in $C_2$-spectra with twisted $\t$-action and let $E^r(R)$ be its homological parametrized homotopy fixed point spectral sequence. Then for each element $x \in E^{2r}_{0,V}(R) \subset H_V(R)$, we have the relation
$$d^{2r}( Q^{k\rho-\epsilon}(x)) =  Q^{k\rho-\epsilon} (d^{2r}(x))$$
for every positive integer $k$ and $\epsilon \in \{0,1\}$. 
\end{proposition}

\begin{proof}
The proof is similar to that of \cite[Proposition~4.2]{BR05}, although there are a few points worth mentioning. In virtue of the $C_{2}$-$\mathbb{E}_{\infty}$-algebra structures we built, the forgetful map $\varphi$ in (\ref{eqn:phi}) and the canonical $C_2$-equivariant map $\nu$ in (\ref{eqn:nu}) are both maps of $C_2$-$\mathbb{E}_\infty$-algebras, i.e.,  they commute with the structure maps described in the previous proof, and thus are compatible with equivariant Dyer--Lashof operations. Second, the Cartan formula holds for Wilson's equivariant Dyer--Lashof operations (\cite[Cor. 1.3.2]{Wil19}), so it is still applicable without change. Finally, we may identify the action of the Dyer--Lashof operations on the Bredon homology of the Spanier--Whitehead dual of $S(\c^r)_+ \cong S^{n\rho-1}_+$ with the action of the Steenrod operations on the Bredon cohomology of $S^{n\rho-1}_+$. Although the Steenrod operations act nontrivially on these cohomology groups in general, they act trivially on the fundamental class, which is the only element relevant to the proof. 
\end{proof}

\begin{remark}
As a consequence of \cref{prop:diffsoperations}, when $r=1$, we deduce compatibility of the Dyer--Lashof operations with the $\bar{\sigma}$-operator
$Q^V(\bar{\sigma}x) = \bar{\sigma}Q^V x$.
\end{remark}

\section{Infinite cycles}\label{Sec:Inf}

In this section we produce infinite cycles in the homological parametrized homotopy fixed point spectral sequence of $C_2$-$\mathbb{E}_\infty$-algebras using equivariant power operations. This extends \cite[Sec. 5]{BR05} to the Real setting. 

\begin{thm}\label{thm:pc}
Let $R$ be a $C_2$-$\mathbb{E}_\infty$-algebra in the $C_2$-category of $C_2$-spectra with twisted $\t$-action. 
Suppose that $x \in H_V(R;\mft)$ survives to the $E^{2r}$-term $E^{2r}_{0,V} \subset H_V(R;\mft)$ of the homological parametrized homotopy fixed point spectral sequence for $R$, and there is a nontrivial differential $d^{2r}(x) = \bar{y}^r \cdot \delta x$.
\begin{enumerate}
	\item If $|x| = k\rho$, then the $2r$ classes
	\[Q^{k\rho}(x), \ Q^{k\rho+\sigma}(x), Q^{(k+1)\rho}(x), \ \ldots, \ Q^{(k+r-1)\rho}(x), \ \text{ and } \ Q^{(k+r-1) \rho+\sigma}(x) + x \delta x\]
	all survive to the $E^\infty$-term.
	\item If $|x| = k\rho+1$, then the $2r$ classes
	\[Q^{k\rho+\sigma}(x), Q^{(k+1)\rho}(x), Q^{(k+1)\rho+\sigma}(x), \ \ldots, \ Q^{(k+r-1)\rho + \sigma}(x), \ \text{ and } \ Q^{(k+r)\rho}(x) + x \delta x\]
	all survive to the $E^\infty$-term.
\end{enumerate}
\end{thm}

\begin{proof}
We adapt the proof of \cite[Thm. 5.1]{BR05} which considers a universal example. 

Recall from \cref{Sec:Diff} that a class $x \in E^{2r}_{0,V}$ is represented by a $C_2$-twisted $\t$-equivariant map $x\colon  S(\c)_+ \wedge S^V \to H\mft \wedge R$ that admits a $C_2$-twisted $\t$-equivariant extension $x' \colon  S(\c^r)_+ \wedge S^V \to H\mft \wedge R$. Let 
\[ X = \bar{D}_2(S(\c^r)_+ \wedge S^V) \coloneq (N_{C_2}^{\mu_2 \times C_2} ( S(\c^r)_+ \wedge S^V))_{h_{C_2}\mu_2} \]
denote the second equivariant extended power of the $C_2$-spectrum $S(\c^r)_+ \wedge S^V$. Here, $N_{C_2}^{\mu_2 \times C_2}$ is the HHR norm from genuine $C_2$-spectra to genuine $(\mu_2 \times C_2)$-spectra, where $\mu_2 \times C_2$ is the Klein four group, and $(-)_{h_{C_2}\mu_2}$ is the $C_2$-parametrized $\mu_2$-homotopy orbits. 

Let us write $H_\star(S(\c^r)_+ \wedge S^V;\mft) = H_\star \{ x, \delta x\}$ where $|x| = V$ and $|\delta x| = V + r \rho - 1$. Then \cite[Thm. 2.15]{Wil17} implies that if $|x| = k\rho$
\[
H_\star(X;\mft) \cong H_\star \left\{ x \delta x,\ Q^{i\rho+\epsilon \sigma}(x), \ Q^{j\rho-\epsilon}(\delta x) \mid  i \geq k, \  j \geq k+r, \ \epsilon \in \{0,1\} \right\},
\] 
while  \cite[Proposition~2.4.1]{Wil19} implies that if $|x| = k\rho+1$, then 
\[H_\star(X;\mft)\cong H_\star \left\{x \delta x, \ Q^{i\rho-\epsilon}(x), Q^{j\rho+\epsilon\sigma}(\delta x) \mid  i\geq k+1, j\geq k+r,\epsilon\in\{0,1\} \right\}.\]
Since Wilson states this last computation without proof, we refer the reader to the proof of \cite[Proposition~5.6]{CGPrealBP} for a sketch of the argument. 

The $C_2$-twisted $\mathbb{T}$-equivariant extension $x'$ induces a $C_2$-twisted $\mathbb{T}$-equivariant map
\[
\bar{D}_2(x')\colon  X=\bar{D}_2(S(\mathbb{C}^r)_+ \wedge S^V) \to \bar{D}_2(H\mft \wedge R).
\]
Observe that the restriction of the $O(2)$-space $E_{C_2}\t$ to $(\mu_2 \times C_2)$-spaces is a model for $E_{C_2}\mu_2$, which is a model for the second arity in an $E_{\infty \rho}$-operad \cite{GM17}. Together with the map above, this yields a $C_2$-twisted $\mathbb{T}$-equivariant structure map 
\[
\xi_2\colon  \bar{D}_2(H\mft \wedge R) \to H\mft \wedge R
\]
that extends the multiplication on $H\mft \wedge R$. We then get a $C_2$-twisted $\mathbb{T}$-equivariant map
\[
H\mft \wedge \bar{D}_2(S(\mathbb{C}^r)_+ \wedge S^V) \xrightarrow{1 \wedge \bar{D}_2(x')} H\mft \wedge \bar{D}_2(H\mft \wedge R) \xrightarrow{1 \wedge \xi_2} H\mft \wedge H\mft \wedge R \xrightarrow{m \wedge 1} H\mft \wedge R 
\]
where $m$ is the multiplication on $H\mft$. Applying $\pi_\star^{C_2}(-)$, we have 
\[
H_{\star}(X; \mft) = H_{\star}( \bar{D}_2(S(\mathbb{C}^r)_+ \wedge S^V); \mft) \to H_{\star}(R; \mft),
\]
which takes the classes generating $H_{\star}(X; \underline{\mathbb{F}}_2)$ to the classes with the same names in $H_{\star}(R; \underline{\mathbb{F}}_2)$. Now $X= \bar{D}_2(S(\mathbb{C}^r)_+ \wedge S^V)$ is a $C_2$-twisted $\mathbb{T}$-equivariant retract of the free $C_2$-commutative algebra $\mathbb{P}(X)$ (see, for example \cite[Sec. 2.1]{Wil17} and \cite[Appendix A.1]{BlHi15}). It follows that the $\mathrm{HHFPSS}_{C_2}(X)$ is a direct summand of the $\mathrm{HHFPSS}_{C_2}(\mathbb{P}(X))$. Thus the formula from \cref{prop:diffsoperations} for the $d^{2r}$-differentials in the spectral sequence for $\mathbb{P}(X)$ also applies to the spectral sequence for $X$. 

We now look at the $\mathrm{HHFPSS}_{C_2}$ for $X =  \bar{D}_2(S(\mathbb{C}^r)_+ \wedge S^V)$. We will show that the classes listed in the statement of the proposition are infinite cycles in this spectral sequence. By naturality of the homological parametrized homotopy fixed point spectral sequence with respect to the map $H\mft \wedge X \to H\mft \wedge R$ we then conclude that the $2r$ target classes listed in $E^{2r}_{0,V}$ are also infinite cycles. 

Observe that the Weyl action on $H_\star(X;\mft)$ is trivial, and so the homological parametrized homotopy fixed point spectral sequence for $X$ has $E^2$-page given by a module over the polynomial algebra $H_\star[\bar{y}]$. We consider the case $|x| = k\rho$ first, where we have
\[E^2_{\ast,\star} = H_\star[\bar{y}] \left\{ x \delta x, \ Q^{i\rho+\epsilon \sigma}(x), \ Q^{j\rho-\epsilon}(\delta x) \mid  i \geq k, \  j \geq k+r, \ \epsilon \in \{0,1\} \right\}.\]
After running the differentials \[d^{2r}(x \delta x) = \bar{y}^r \cdot (\delta x)^2 = \bar{y}^r \cdot Q^{(k+r-1)\rho+\sigma}(\delta x) \quad\text{ and }\quad d^{2r}(Q^{i\rho+\epsilon \sigma} (x)) = \bar{y}^r \cdot Q^{i\rho+\epsilon \sigma}(\delta x)\] for all $i \geq k$ and $\epsilon \in \{0,1\}$, together with their $\bar{y}$-multiples, we have
\[E^{2r+2}_{\ast,\star} = H_\star[\bar{y}] \left\{x\delta x+Q^{(k+r-1)\rho+\sigma}(x), Q^{(k+r-1)\rho}(x), Q^{i\rho+\epsilon\sigma}(x) \mid  k \leq i \leq k+r-2, \ \epsilon \in \{0,1\} \right\}\]
plus some $\bar{y}$-torsion classes from $E^2_{\ast,\star}$ in filtrations $-2r < s \leq 0$. All further differentials on the $H_\star$-module generators of $E^{2r+2}_{0,\star}$ are zero since the target groups vanish. Thus all further differentials from the vertical axis are zero and the spectral sequence collapses at $E^{2r+2} = E^\infty$.


In the case of $|x| = k\rho+1$, we instead have 
\[E^2_{\ast,\star} = H_\star[\bar{y}] \left\{x \delta x, Q^{i\rho-\epsilon}(x),Q^{j\rho+\epsilon\sigma}(\delta x)\mid  i \geq k+1, j\geq k+r,\epsilon\in\{0,1\} \right\}.\]
After running the differentials \[d^{2r}(x\delta x) = \bar{y}^r\cdot (\delta x)^2 = \bar{y}^r\cdot Q^{(k+r)\rho}(\delta x)\quad\text{ and }\quad  d^{2r}(Q^{i\rho-\epsilon}(x)) = \bar{y}^r\cdot Q^{i\rho-\epsilon}(\delta x)\]
for all $i\geq k+1$ and $\epsilon\in\{0,1\}$, we have 
\[E^{2r+2}_{\ast,\star} = H_\star[\bar{y}] \left\{x\delta x + Q^{(k+r)\rho}(x), Q^{(k+r)\rho-1}(x), Q^{i\rho-\epsilon}(x) \mid  k+1\leq i\leq k+r-1, \epsilon\in\{0,1\} \right\}\]
plus some $\bar{y}$-torsion classes from $E^2_{\ast,\star}$ in filtrations $-2r < s \leq 0$. As before, the spectral sequence collapses at $E^{2r+2} = E^\infty$.
\end{proof}

\begin{remark}
The restriction $V = k\rho$ or $k\rho+1$ for this theorem was required to apply Wilson's analysis of the genuine extended powers for $S^{k\rho+1}, S^{k\rho}$ and $S^{k\rho-1}$. 
Given the behavior of the homology of $\bar{D}_2(S^{a+b\sigma})$ for $|a-b|>1$ alluded to at the end of \cite{Wil19}, it is unclear to the authors how to generalize this theorem for classes in other degrees.  
\end{remark}

\begin{figure}[H]
\begin{center}
\includegraphics[scale = 0.75]{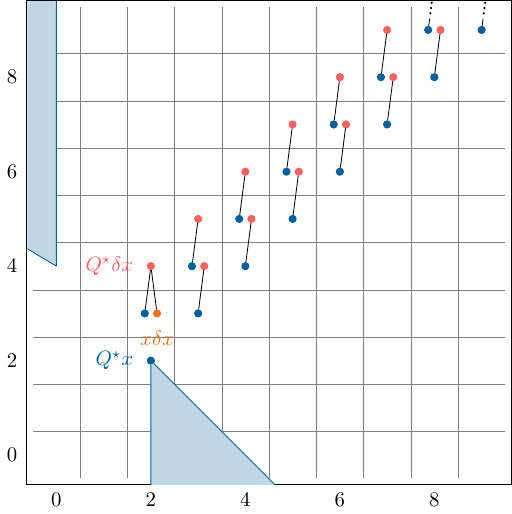}\qquad 
\includegraphics[scale = 0.75]{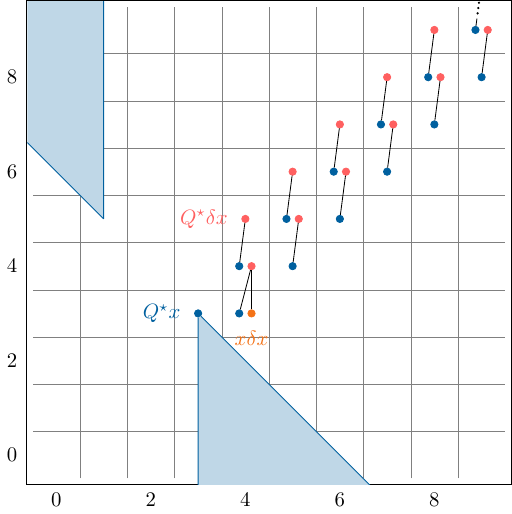}
\end{center}
\caption{The homological parametrized homotopy fixed point spectral sequence for 
$\bar{D}_2(S(\mathbb{C})_+\wedge S^{\rho})$ (left) and $\bar{D}_2(S(\mathbb{C})_+\wedge S^{\rho+1})$ (right). 
A single $(x,y)$ coordinate corresponds to the $RO(C_2)$-degree $x+y\sigma$ with filtration degree suppressed.
Each dot represents a copy of $H_{\star}$. A $d^{2r}$ differential goes up by $r$ and right by $r-1$. }
\end{figure}

\section{Computations}\label{Sec:Compute}

Our \cref{thm:pc} shows that differentials in the $\mathrm{HHFPSS}_{C_2}$ for $C_2$-$\mathbb{E}_\infty$-rings give rise to permanent cycles by means of equivariant power operations.
In this section, we apply this theorem to compute homological information for Real trace invariants of various rings. As in the classical case \cite{BR05}, we find that 
\cref{thm:pc} can be used to completely compute the spectral sequence in many examples of interest. 

\begin{notation}
To lighten the notation, in this section we will write 
\[T^\Phi(A)\coloneq\THR(A)^{\Phi C_2}\] 
for the geometric fixed points of Real topological Hochschild homology.
We also suppress the coefficient Mackey functor from the notation and write
$H_\star(X)=H_\star(X;\mft)$ for the $RO(C_2)$-graded Bredon homology of $X$.
\end{notation}
\subsection{The continuous homology of $\TCR^{-}$ of Real bordism}\label{Sec:MUR}

The goal for this subsection is to compute the $E^\infty$-terms of the 
$\mathrm{HHFPSS}_{C_2}$ for $\MUR$ and $\BPR$. 
We begin by computing the homology of $\THR(\MUR)$. 

\begin{proposition}\label[proposition]{prop:HTHRMUR}
The $\mft$-homology of $\THR(\MUR)$ is given by 
\[H_\star \THR(\MUR) \cong H_{\star}[\bar{b}_i \mid  i\geq 1]\otimes_{H_\star} \Lambda_{H_\star}[\bar{\sigma}\bar{b}_i \mid i \geq 1],\]	
with $|\bar{b}_i| = i\rho$ and $|\bar{\sigma}\bar{b}_i| = i\rho+\sigma$. 
\end{proposition}
\begin{proof}
Since $H\underline{\F}_2$ is Real-oriented, we have 
\[H_\star\MUR\cong H_{\star}[\bar{b}_i \mid i\geq 1],\] with $|\bar{b}_i| = i\rho$. Let $\bar{b}_i\in H_\star\THR(\MUR)$ be the image of the 
element with the same name under the map $H_\star\MUR\to H_\star\THR(\MUR)$ induced by the left unit. On underlying, we have 
\[H_\ast\THH(\mathrm{MU})\cong \F_2[b_i \mid i\geq 1]\otimes\Lambda_{\F_2}[\sigma b_i \mid i\geq 1],\] with 
$b_i$ and $\sigma b_i$ the restrictions of $\bar{b}_i$ and $\bar{\sigma}\bar{b}_i$, respectively.

We now check the result on reduced geometric fixed points. First note that 
\[T^{\Phi}(\MUR)\simeq \mathrm{MO}\wedge_{MU} \mathrm{MO}.\]
The arguments of \cite[Sec.~5.2.2]{HHR} show that $\overline{\Phi}(\bar{b}_i)$ are polynomial generators for $H_\ast\mathrm{MO}$. Since restriction is injective on 
$H_{\ast\rho}\MUR$, we can deduce the norms $n(b_i) = \bar{b}_i^2$. Then by \cref{prop:gfpmodulestructure} we then have that $H_\ast\mathrm{MU}\to H_\ast \mathrm{MO}$ sends $b_i\mapsto \overline{\Phi}(\bar{b}_i)^2$. 
It follows that the homology $H_\ast T^\Phi(\MUR)$ can be computed by a collapsing K\"unneth spectral sequence.
Now \cref{cor:residualmugfp} implies that $\overline{\Phi}(\bar{b}_i)$ and 
$\overline{\Phi}(\bar{b}_i + \bar{\sigma}\bar{b}_i)$ are the images of the generators of $H_\ast(\mathrm{MO})$ under the left and right units $\mathrm{MO}\to T^\Phi(\MUR)$, respectively. 
Since $H_\ast\mathrm{MU}\to H_\ast \mathrm{MO}$ sends $b_i\mapsto \overline{\Phi}(\bar{b}_i)^2$, we must have that
$\overline{\Phi}(\bar{\sigma}\bar{b}_i)^2 = 0$, and so 
\[H_\ast T^{\Phi}(\MUR)\cong \F_2\left[\overline{\Phi}(\bar{b}_i) \mid i>0 \right]\otimes \Lambda_{\F_2}\left[ \overline{\Phi}(\bar{\sigma}\bar{b}_i) \mid i>0 \right].\]

The proof of the additive statement now follows from \cref{lemma:BWLemma}. More precisely, we use the lemma to establish that the monomial basis 
generated by the $\bar{b}_i, \bar{\sigma}\bar{b}_i$ is a free $H_\star$-basis. 

It remains to check that the $\bar{\sigma}\bar{b}_i$ square to zero. This can be checked using equivariant power operations:
\[(\bar{\sigma}\bar{b}_i)^2 = Q^{i\rho + \sigma}(\bar{\sigma}\bar{b}_i) = \bar{\sigma}(Q^{i\rho+\sigma}(\bar{b}_i)) = 0,\] where the last equality follows from the fact that $Q^{i\rho+\sigma}\bar{b}_i = 0\in H_\star\MUR$ for degree reasons. 
\end{proof}

\begin{remark}
We note that the form of $H_\star\THR(\MUR)$ and $H_\ast T^{\Phi}(\MUR)$ can also be deduced from 
the formula $\THR(\MUR)\simeq \MUR\wedge\Sigma_+^\infty B^\sigma BU_{\mathbb{R}}$ given in \cite[Corollary~7.2.2]{HKZ20}. We opted to give a detailed proof here that shows that the 
exterior generators could be taken to be the images under $\bar{\sigma}$ of the polynomial generators. 
\end{remark}

\begin{thm}\label{thm:HcTHRMUR}
The $E^\infty$-term of the $\mathrm{HHFPSS}_{C_2}(\THR(\MUR))$ converging to the continuous homology, $H^c_{\star}(\TCR^-(\MUR))$, is given by
\[E^\infty_{\ast,\star} \cong H_\star[\bar{y}, \bar{b}_i^2 \mid i>0] \otimes_{H_\star} \Lambda_{H_\star}[\bar{b}_i\bar{\sigma}\bar{b}_i \mid i>0] \oplus (\text{simple $\bar{y}$-torsion}).\]
\end{thm}

\begin{proof}
%
%
%

The spectral sequence has the form
\[E^2_{\ast,\star} = H_\star\langle\bar{y}\rangle[\bar{b}_i \mid i>0] \otimes_{H_\star} \Lambda_{H_\star}[\bar{\sigma}\bar{b}_i \mid i>0] \Rightarrow H^c_\star \TCR^-(\MUR),\]
and by \cref{prop:d2formula}, the $d^2$-differentials are generated 
by $d^2(\bar{b}_i) = (\bar{\sigma}\bar{b}_i)\bar{y}$. 
%
This leaves $$E^3 = E^4 = H_\star[\bar{y},\bar{b}_i^2 \mid i>0] \otimes_{H_\star} \Lambda_{H_\star}[\bar{b}_i\bar{\sigma}\bar{b}_i \mid i>0] \oplus (\text{simple $\bar{y}$-torsion}).$$
We note that the $E^4$-page is commutative since the $d^2$-differentials are all of the form $d^2(x) = \bar{y}\bar{\sigma}x$ and $\bar{y}x = \bar{\sigma}x\bar{y}$. 
It then follows from \cref{thm:pc} that  $\bar{b}_i^2 = Q^{i\rho}(\bar{b}_i)$ and $Q^{i\rho+\sigma}(\bar{b}_i)+\bar{b}_i\bar{\sigma}\bar{b}_i = \bar{b}_i\bar{\sigma}\bar{b}_i$ for $i>0$ are permanent cycles. 
\end{proof}

We now move on to consider the computation for $\BPR$. It is important to note that $\BPR$ is an $\mathbb{E}_\rho$-algebra \cite[Theorem~7.3.1]{QuinnZhuMultiplicative} so that $\THR(\BPR)$ is defined.

\begin{proposition}\label[proposition]{prop:HTHRBPR}
The $\mft$-homology of $\THR(\BPR)$ is given by
\[H_\star\THR(\BPR)\cong H_\star[\bar{\xi}_i \mid i>0]\otimes_{H_\star}\Lambda_{H_\star}[\bar{\sigma}\bar{\xi}_i \mid i>0],\]
where $|\bar{\xi}_i| = (2^i-1)\rho$ and $|\bar{\sigma}\bar{\xi}_i| = (2^i-1)\rho+\sigma$.
Further,
$H_\star\THR(\MUR)\to H_\star\THR(\BPR)$ is a surjection sending $\bar{b}_{2^i-1}\mapsto \bar{\xi}_i$.
\end{proposition}
\begin{proof}
Define $\bar{\xi}_i\in H_\star\THR(\BPR)$ as the image of $\bar{b}_{2^i-1}$ under 
\[H_\star\THR(\MUR)\to H_\star\THR(\BPR).\]
The map induced on underlying homology is the surjection 
\[ \F_2 \left[b_i \mid i>0 \right] \otimes \Lambda_{\F_2} \left[ \sigma b_i \mid i>0 \right] \to \F_2 \left[ \bar{\zeta}_i^2 \mid i>0 \right] \otimes\Lambda_{\F_2} \left[ \sigma(\bar{\zeta}_i^2) \mid i>0 \right] \]
sending $b_{2^i-1}\mapsto \bar{\zeta}^2_i$, $\sigma b_{2^i-1}\mapsto \sigma (\bar{\zeta}^2_i)$, and all other generators to zero.  Then, by construction, the $\bar{\xi}_i$ restrict to $\bar{\zeta}_i^2$ and $\bar{\sigma}\bar{\xi}_i$ restricts 
to $\sigma(\bar{\zeta}_i^2)$. 
It then suffices to check that the geometric fixed points of the $\bar{\xi}_i$ and $\bar{\sigma}\bar{\xi}_i$ generate
$H_\ast T^{\Phi}(\BPR)$. 

We proceed as in the proof of \cref{prop:HTHRMUR}. 
First, observe that we have an equivalence $T^{\Phi}(\BPR)\simeq H\F_2\wedge_{\mathrm{BP}} H\F_2,$ and that the resulting K\"unneth spectral sequence computing $H_\ast T^{\Phi}(\BPR)$ collapses.
The geometric fixed points $\overline{\Phi}(\bar{\xi}_i)$ and $\overline{\Phi}(\bar{\sigma}\bar{\xi}_i)$ then generate this homology as an algebra.
The map $H_\ast\mathrm{BP}\to H_\ast H\F_2$ is the embedding of the subalgebra generated by $\bar{\zeta}_i^2$, implying that $\overline{\Phi}(\bar{\sigma}\bar{\xi}_i)^2 = 0$. 
\end{proof}

\begin{thm}\label{thm:HcTHRBPR}
The $E^\infty$-term of the $\mathrm{HHFPSS}_{C_2}(\THR(\BPR))$ converging to the continuous homology, $H_\star^c(\TCR^{-}(\BPR))$, is given by 
\[H_\star \left[\bar{y},\bar{\xi}_i^2 \mid i>0 \right] \otimes_{H_\star}\Lambda_{H_\star} \left[\bar{\xi}_i\bar{\sigma}\bar{\xi}_i \mid i>0\right] \oplus (\text{simple $\bar{y}$-torsion}).\]
\end{thm}

\begin{proof}
By \cref{prop:HTHRBPR}, the map induced on $E^2$-pages by $\THR(\MUR)\to\THR(\BPR)$ is surjective, and we have the differentials $d^2(\bar{\xi}_i)=\bar{y}\cdot\bar{\sigma}\bar{\xi}_i.$ The map induced on $E^4$-pages is still surjective, implying collapse at $E^4$.  
\end{proof}

As a corollary, we obtain the following statement for the geometric fixed points. 
\begin{corollary}\label[corollary]{cor:HcgfpTHRbordism}
	The $E^\infty$-term of the $\mathrm{HHFPSS}(T^{\Phi}(\MUR))$ converging to the continuous homology, $H_\ast^cT^{\Phi}(\MUR)^{h\mu_2}$, is given by
	\[\F_2[z] \left[ \overline{\Phi}(\bar{b}_i^2) \mid i>0 \right]\otimes\Lambda_{\F_2} \left[ \overline{\Phi}(\bar{b}_i\bar{\sigma}\bar{b}_i) \mid i>0 \right] \oplus\text{(simple $z$-torsion)}.\]
	The analogous result for $\BPR$ also holds. 
\end{corollary}
\begin{proof}
From \cref{prop:gfpE4pg}, and the arguments above, we readily find the $E^4$-page is as in the theorem statement.
All generators are either $z$-torsion or images of permanent cycles under the map \[\overline{\Phi}\colon \mathrm{HHFPSS}_{C_2}(\THR(\MUR))\to\mathrm{HHFPSS}(T^{\Phi}(\MUR))\] of \cref{prop:gfpssmap}, hence they are permanent cycles.
\end{proof}

\subsection{The continuous homology of $\TCR^{-}(\BPR\langle n\rangle)$}
We now consider the truncated Real Brown--Peterson spectra $\BPR\langle n \rangle$ for 
$-1\leq n < \infty$, where, by convention, $\BPR\langle -1\rangle = H\mft$. 
We begin by gathering some results about $\BPR\langle n\rangle$ and its homology.

A computation by Carrick--Hill--Ravenel recorded in \cite[Proposition~4.3]{LPTsplitting} shows that the quotient map $\BPR\langle n\rangle\to H\mft$ induces an inclusion 
\[H_\star\BPR\langle n\rangle \cong \frac{H_\star \left[ \bar{\xi}_i,\bar{\tau}_{n+i} \mid i\geq 1 \right]}{\bar{\tau}_j^2 = u_\sigma\bar{\xi}_{j+1} + a_\sigma\bar{\tau}_{j+1}}\subset\mathcal{A}_\star^{C_2},\]
just as in the underlying homology we have
\[H_\ast\mathrm{BP}\langle n\rangle\cong\F_2 \left[\bar{\zeta}_1^2, \ldots ,\bar{\zeta}_n^2,\bar{\zeta}_{n+1},\ldots\right] \subset\mathcal{A}_\ast.\]
The geometric fixed points of $\BPR\langle n\rangle$ are equivalent to the quotient $H$-module
\[\BPR\langle n\rangle^{\Phi C_2}\simeq (a_\sigma^{-1}\BPR\langle n\rangle)^{C_2} \simeq H[b_n],\] where $b_{-1} = b$ is represented by $u_\sigma/a_\sigma\in a_\sigma^{-1}H_\star$, and the
induced map $H[b_n]\to H[b]$ sends $b_n\mapsto b^{2^{n+1}}$.  
\begin{proposition}\label[proposition]{prop:resgfpHBPRn}
The restriction and reduced geometric fixed point maps for $H_\star\BPR\langle n\rangle$ are given as follows:
\begin{enumerate}[(a)]
	\item $\res\colon H_\star\BPR\langle n\rangle\to H_\ast \mathrm{BP}\langle n\rangle$ sends $\begin{cases}\bar{\xi}_i\mapsto \bar{\zeta}_i^2 \\ \bar{\tau}_{n+i}\mapsto\bar{\zeta}_{n+i+1}\end{cases} (i\geq 1)$,\quad and \\
	\item $\overline{\Phi}\colon H_\star\BPR\langle n\rangle\to \mathcal{A}_\ast[b_n]$ sends $\begin{cases} \bar{\xi}_i\mapsto \bar{\zeta}_i \\ \bar{\tau}_{n+i}\mapsto b_n^{2^{i-1}}\end{cases} (i\geq 1)$. 
\end{enumerate}
\end{proposition}
\begin{proof}
Since inclusion maps $H_\star\BPR\langle n\rangle \to \mathcal{A}_\star^{C_2}$ remain injective upon restricting to underlying and taking geometric fixed points, it suffices to check the proposition for $\mathcal{A}_\star^{C_2}$.
The claim about restriction follows from \cite[Proposition~4.10]{CGPrealBP}. The claim about geometric fixed points 
is proved modulo $b_n$ in \cite[Proposition~4.11]{CGPrealBP}. In particular, we have $\overline{\Phi}(\bar{\xi}_i) = \bar{\zeta}_i\mod b_n$. However, since $\bar{\xi}_i$ is in the image of 
$H_\star\BPR\to H_\star\BPR\langle n\rangle$, $\overline{\Phi}(\bar{\xi}_i)$ must be in the image of $\mathcal{A}_\star^{C_2}\to\mathcal{A}_\star^{C_2}[b_n]$. 

It remains to compute $\overline{\Phi}(\bar{\tau}_{n+i})$. For this we consider the commutative diagram 

\begin{center}
\begin{tikzcd}
H_\star\ar[d,"\eta_R"]\ar[r,"\Phi_S"]& \F_2[b]\ar[d] \\ 
\mathcal{A}_\star^{C_2} \ar[r,"\overline{\Phi}_H"]& \mathcal{A}_\ast[b].
\end{tikzcd}
\end{center}
We emphasize that $\Phi_S$ is the unreduced geometric fixed points map.
Using the formula for the right unit of $u_\sigma$ in $(H_\star,\mathcal{A}_\star^{C_2})$ given in \cite[Proposition~6.37]{HK01}, we have  \[\overline{\Phi}_H(\eta_R(u_\sigma)) = \overline{\Phi}_H(u_\sigma + a_\sigma\bar{\tau}_0)=\overline{\Phi}(\bar{\tau}_0).\]
On the other hand, $\Phi_S(u_\sigma) = b_{-1}$. Thus $\overline{\Phi}(\bar{\tau}_0) = b_{-1}$. From the relation $\bar{\tau}^2_i = u_\sigma\bar{\xi}_{i+1} + a_\sigma\bar{\tau}_{i+1}$, we find that
$\overline{\Phi}(\bar{\tau}_{i}) = b^{2^{i}}$ for $\bar{\tau}_i\in\mathcal{A}_\star^{C_2}$. Then $\overline{\Phi}(\bar{\tau}_{n+i}) = b_{n}^{2^{i-1}}$ for $\bar{\tau}_{n+i}\in H_\star\BPR\langle n\rangle$. 
\end{proof}

\begin{warning}
In what follows we assume that $\BPR\langle n\rangle$ is a $C_2$-$\mathbb{E}_\infty$-ring. Such forms of $\BPR\langle n\rangle$ only exist for 
$-1\leq n \leq 2$. Nevertheless, it is expected that the arguments of Hahn--Wilson can be generalized to construct $\mathbb{E}_{2\sigma+1}$-forms of $\BPR\langle n\rangle$ for all $n\geq -1$ (see \cite[Remark~1.0.14]{HW20}) and we expect that some of our results can be generalized to include these examples. 
\end{warning}

\begin{proposition}\label[proposition]{prop:HTHRBPRnPhiC2}
When $-1 \leq n \leq 2$, there is an $\mu_2$-equivariant isomorphism 
\[H_\ast T^{\Phi}(\BPR\langle n\rangle) \cong \frac{H_\ast\mathrm{BP}\langle n\rangle \left[ \mu_2\cdot h_1,\ldots ,\mu_2\cdot h_{n+1} \right] \left[\mu_2\cdot  b_{n}\right]}{h_i^2 = (\mu h_i)^2 = \bar{\zeta}_i^2},\]
where $\mu_2\cdot x = \{x,\mu x\}$ denotes a $\mu_2$-orbit, and $| h_i| = 2^i-1, | b_{n}| = 2^{n+1}$. 
\end{proposition}
\begin{proof}
Recall that $\BPR\langle n\rangle^{\Phi C_2}\simeq H[ b_{n}]$, where $| b_{n}| = 2^{n+1}$. Given the assumed ring structure, we have
\[T^{\Phi}(\BPR\langle n\rangle) \simeq H[ b_{n}]\wedge_{\mathrm{BP}\langle n\rangle}H[ b_{n}].\]
By \cref{prop:gfpmodulestructure}, the induced map $H_\ast\mathrm{BP}\langle n\rangle\to\mathcal{A}_\ast[b_n]$ is determined by 
$\overline{\Phi}(n(\bar{\zeta}_i))$  for $i>n+1$ and $\overline{\Phi}(n(\bar{\zeta}_j))$ for $1\leq j\leq n+1$. Since the $C_2$-action on 
$H_\ast\mathrm{BP}\langle n\rangle$ is trivial, we must have $\res(n(\bar{\zeta}_i))=\bar{\zeta}_i^2$. It follows that 
$n(\bar{\zeta}_i) = \bar{\xi}_i\mod a_\sigma$ in $H_{(2^i-1)\rho}\BPR\langle n\rangle$. For degree reasons, any class $x$ such that $n(\bar{\zeta}_i) = \bar{\xi}_i+a_\sigma x$ must be in the negative cone and so must be killed by some power of 
$a_\sigma$. Thus $\overline{\Phi}(n(\bar{\zeta}_i)) = \overline{\Phi}(\bar{\xi}_i) = \bar{\zeta}_i$. Similarly, we can deduce that 
$\overline{\Phi}(n(\bar{\zeta}_j^2)) = \bar{\zeta}_j^2$.  
The homology can then be computed by a collapsing K\"unneth spectral sequence: $H_\ast T^{\Phi}(\BPR\langle n\rangle)\cong \mathcal{A}_\ast[ b_{n}]\otimes_{H_\ast\mathrm{BP}\langle n\rangle}\mathcal{A}_\ast[ b_{n}]$.
The result then follows by setting $ h_i= \bar{\zeta}_i\otimes 1$, $\mu h_i = 1\otimes\bar{\zeta}_i$ for $1\leq i\leq n+1$, and similarly $b_{n} =  b_{n}\otimes 1$, $\mu b_n = 1\otimes b_n$. 
\end{proof}

\begin{proposition}\label[proposition]{prop:HTHRBPRn}
	For $-1 \leq n \leq 2$, the $\mft$-homology of $\THR(\BPR\langle n\rangle)$ is given by 
	\[H_\star\THR(\BPR\langle n\rangle)\cong H_\star(\BPR\langle n\rangle)\otimes_{H_\star}\Lambda_{H_\star} \left[\sigma\bar{\xi}_1,\ldots ,\bar{\sigma}\bar{\xi}_{n+1} \right] \otimes_{H_\star}H_\star[\bar{x}_n],\] where $\bar{x}_n = \bar{\sigma}\bar{\tau}_{n+1}$.
\end{proposition}
\begin{proof}
The proof is similar to that of \cref{prop:HTHRMUR}.  We define classes $\bar{\xi}_i,\bar{\tau}_{n+i}\in H_\star\THR(\BPR\langle n\rangle)$ for $i\geq 1$ as the images under the left unit map, and consequently $\bar{\sigma}\bar{\xi}_i, \bar{x}_n=\bar{\sigma}\bar{\tau}_{n+1}$ then form a corresponding monomial basis $B$. 
By \cref{prop:resgfpHBPRn}, \cite[Sec.~6]{BR05}, and \cref{prop:HTHRBPRnPhiC2}, the restriction and geometric fixed points of $B$ forms a basis of $H_\ast\THH(\BPR\langle n\rangle)$ and $H_\ast T^\Phi(\BPR\langle n\rangle)$. 
Thus, \cref{lemma:BWLemma} implies that $B$ indeed forms an $H_\star$-basis. 
Compatibility of $\bar{\sigma}$ with equivariant power operations then implies \[(\bar{\sigma}\bar{\xi}_j)^2 = Q^{2^j\rho+\sigma}(\bar{\sigma}\bar{\xi}_j) = \bar{\sigma}Q^{2^j\rho+\sigma}(\bar{\xi}_j) = 0,\] where the last equality follows from the fact that $Q^{2^j\rho+\sigma}\bar{\xi}_j = 0$ in $H_\star\BPR\langle n\rangle$ for degree reasons. 
\end{proof}

\begin{thm}\label{thm:HcTHRBPRn}
For $-1 \leq n \leq 2$, the $E^\infty$-term of the $\mathrm{HHFPSS}_{C_2}(\THR(\BPR\langle n\rangle))$ converging to the continuous 
homology $H_\star^c\TCR^{-}(\BPR\langle n\rangle)$ is given by 
\[E^\infty_{\ast,\star}\cong \frac{H_\star \left[\bar{\xi}_i',\bar{\tau}_i' \mid i\geq n+2 \right]}{(\bar{\tau}_j')^2 = u_\sigma\bar{\xi}_{j+1}'+a_\sigma\bar{\tau}_{j+1}'}\otimes_{H_\star}H_\star \left[\bar{\xi}_1^2,...,\bar{\xi}_{n+1}^2,\bar{y}\right] \otimes_{H_\star}\Lambda_{H_\star} \left[\bar{\xi}_1\bar{\sigma}\bar{\xi}_1,...,\bar{\xi}_{n+1}\bar{\sigma}\bar{\xi}_{n+1} \right]\] plus a simple $\bar{y}$-torsion module in filtration zero.
\end{thm}
\begin{proof}
By \cref{prop:HTHRBPRn}, the $E^2$-page is identified with 
\[E^2_{\ast,\star}\cong \frac{H_\star \left[ \bar{\xi}_i,\bar{\tau}_{n+i} \mid i\geq 1 \right]}{\bar{\tau}_j^2 = u_\sigma\bar{\xi}_{j+1} + a_\sigma\bar{\tau}_{j+1}}\otimes_{H_\star}\Lambda_{H_\star} \left[\sigma\bar{\xi}_1,\ldots,\bar{\sigma}\bar{\xi}_{n+1}\right] \otimes_{H_\star}H_\star[\bar{x}_n]\langle \bar{y}\rangle\]
The nontrivial $d^2$-differentials follow from \cref{prop:d2formula}:
\[d^2\bar{\xi}_j = \bar{y}\cdot\bar{\sigma}\bar{\xi}_j,\quad 1\leq j\leq n+1,\quad\text{ and }\quad d^2\bar{\tau}_{n+i} = \bar{y}\cdot\bar{x}_n^{2^{i-1}},\quad i\geq 1.\]
For degree reasons $d^2(\bar{\xi}_{n+2}) = 0$, and so compatibility with power operations implies 
$d^2(\bar{\xi}_{n+i}) = 0$ for $i\geq 2$. 

Since the inclusion $H_\star\BPR\langle n\rangle\hookrightarrow \mathcal{A}_\star^{C_2}$ is induced by a map of $C_2$-$\mathbb{E}_\infty$-rings, the action of the 
equivariant power operations on the $\bar{\xi}_i$ and $\bar{\tau}_{n+i}$ follows from \cref{thm:SteenrodPowerOps}. Thus, from \cref{thm:pc} and \cref{thm:SteenrodPowerOps}, we have the following permanent cycles: 
\begin{equation*}
\begin{split}
Q^{(2^j-1)\rho}(\bar{\xi}_j) = \bar{\xi}_j^2,\qquad & Q^{(2^j-1)\rho + \sigma}\bar{\xi}_j + \bar{\xi}_j \bar{\sigma}\bar{\xi}_j = \bar{\xi}_j\bar{\sigma}\bar{\xi}_j, \quad 1\leq j\leq n+1\\
Q^{(2^{n+i}-1)\rho+\sigma}(\bar{\tau}_{n+i}) = \bar{\xi}_{n+i+1},\qquad & Q^{2^{n+i}\rho}\bar{\tau}_{n+i} +\bar{\tau}_{n+i}\bar{\sigma}\bar{\tau}_{n+i} = \bar{\tau}_{n+i+1}+\bar{\tau}_{n+i}\bar{x}_n^{2^{i-1}},\quad i\geq 1,
\end{split}
\end{equation*}
where we have used the fact that for degree reasons $Q^{(2^j-1)\rho+\sigma}\bar{\xi}_j = 0$ in $H_\star\BPR\langle n\rangle$ for $1\leq j\leq n$.
Now set $\bar{\tau}_{n+2}' = \bar{\tau}_{n+2}$, $\bar{\xi}_{n+2}' =\bar{\xi}_{n+2}$, and for $i\geq 3$, 
\[\bar{\xi}_{n+i}' = \bar{\xi}_{n+i} + \bar{\xi}_{n+i-1}\bar{x}_n^{2^{i-2}},\quad\text{ and }\quad \bar{\tau}_{n+i}' = \bar{\tau}_{n+i} + \bar{\tau}_{n+i-1}\bar{x}_n^{2^{i-2}}.\] The relation 
$(\bar{\tau}_j')^2 = u_\sigma\bar{\xi}_{j+1}' + a_\sigma\bar{\tau}_{j+1}'$ can then be checked on the $E^\infty$-page for $j\geq n+2$.
\end{proof}


As a corollary, we can compute the $E^\infty$-page for the continuous homology of $T^\Phi(\BPR\langle n\rangle)^{h\mu_2}$. 
\begin{corollary}\label[corollary]{cor:HcgfpBPRn}
For $-1 \leq n \leq 2$, the $E^\infty$-term of the $\mathrm{HHFPSS}( T^{\Phi}(\BPR\langle n\rangle))$ converging to the continuous homology
$H^c_\ast T^{\Phi}(\BPR\langle n\rangle)^{h\mu_2}$, is given by 
\[\frac{H_\ast\mathrm{BP}\langle n\rangle \left[\bar{\zeta}_1', \ldots ,\bar{\zeta}_{n+1}' \right] \left[z,b_n' \right]}{ \left((\bar{\zeta}_i')^2-\bar{\zeta}_i^4 \right)},\]
plus a simple $z$-torsion module in filtration zero.
\end{corollary}
\begin{proof}
From \cref{prop:HTHRBPRnPhiC2}, we can compute the $E^4$-page, which by \cref{prop:gfpE4pg} is just $H^\ast(\mu_2;H_\ast T^{\Phi}(\BPR\langle n\rangle))$. 
Let $\bar{\zeta}_j' = h_j\mu h_j$ for $1\leq j\leq n$ and $b_n' = b_n\mu b_n$. These are $\mu_2$-invariants in $H_\ast T^{\Phi}(\BPR\langle n\rangle)$ and so give 
rise to elements in $E^4_{0,\ast}$. 
One then finds that the $E^4$-page is precisely as in the theorem statement. We claim that the spectral sequence collapses at $E^4$. 

To see this, it suffices to note that $\overline{\Phi}(\tau_{n+1}') = b_n', \overline{\Phi}(\bar{\xi}_{n+i}) = \bar{\zeta}_{n+i}$ for $i\geq 1$, and \[\overline{\Phi}(\bar{\xi}_j^2) =\bar{\zeta}_j^2,\quad
\overline{\Phi}(\bar{\xi}_j\bar{\sigma}\bar{\xi}_j) = \bar{\zeta}_j^2 + \bar{\zeta}_j',\quad 1\leq j\leq n.\]
Thus all generators on the $E^4$-page are images of permanent cycles under the geometric fixed point map 
\[\overline{\Phi} \colon \mathrm{HHFPSS}_{C_2}(\THR(\BPR\langle n\rangle))\to \mathrm{HHFPSS}( T^{\Phi}(\BPR\langle n\rangle)).\qedhere\] 
\end{proof}

\bibliographystyle{alpha}
\bibliography{master}

\end{document}